\documentclass[11pt]{amsart}
\usepackage{amsmath, amscd, amsthm, amssymb}
\usepackage{graphicx}

\usepackage{fullpage}

\def\C{{\mathbf C}}
\def\R{{\mathbf R}}

\def\Q{{\mathbf Q}}

\def\A{{\mathbf A}}

\makeatletter
\newcommand{\subalign}[1]{%
  \vcenter{%
    \Let@ \restore@math@cr \default@tag
    \baselineskip\fontdimen10 \scriptfont\tw@
    \advance\baselineskip\fontdimen12 \scriptfont\tw@
    \lineskip\thr@@\fontdimen8 \scriptfont\thr@@
    \lineskiplimit\lineskip
    \ialign{\hfil$\m@th\scriptstyle##$&$\m@th\scriptstyle{}##$\crcr
      #1\crcr
    }%
  }
}
\makeatother

\newtheorem{theorem}{Theorem}[section]
\newtheorem{thm}[theorem]{Theorem}
\newtheorem{lemma}[theorem]{Lemma}

\newtheorem{proposition}[theorem]{Proposition}
\newtheorem{prop}[theorem]{Proposition}

\newtheorem{claim}[theorem]{Claim}

\theoremstyle{definition}

\theoremstyle{remark}

\DeclareMathOperator{\diag}{diag}
\renewcommand{\hom}{\operatorname{Hom}}
\renewcommand{\(}{\left(}
\renewcommand{\)}{\right)}
\newcommand{\set}[1]{\left\lbrace#1\right\rbrace}
\newcommand{\floor}[1]{\left\lfloor#1\right\rfloor}

\newcommand{\gen}[1]{{\prec}#1{\succ}}
\newcommand{\ol}[1]{\overline{#1}}

\newcommand{\mm}[4]{\(\begin{smallmatrix} #1 & #2\\ #3 & #4\end{smallmatrix}\)}

\DeclareMathOperator{\ind}{Ind}

\DeclareMathOperator{\Spin}{Spin}
\DeclareMathOperator{\Sp}{Sp}
\DeclareMathOperator{\GSp}{GSp}
\DeclareMathOperator{\PGSp}{PGSp}

\DeclareMathOperator{\SL}{SL}
\DeclareMathOperator{\GL}{GL}
\DeclareMathOperator{\PGL}{PGL}
\DeclareMathOperator{\ord}{ord}

\newcommand{\inj}{\hookrightarrow}

\begin{document}
\title{A multivariate integral representation on $\GL_2 \times \GSp_4$ inspired by the pullback formula}

\author{Aaron Pollack}
\address{Department of Mathematics, Institute for Advanced Study, Princeton, NJ 08540, USA}\email{aaronjp@math.ias.edu}
\author{Shrenik Shah}
\address{Department of Mathematics, Columbia University, New York, NY 10027, USA}\email{snshah@math.columbia.edu}

\thanks{A.P.\ has been supported by NSF grant DMS-1401858.  S.S.\ has been supported by NSF grant DMS-1401967.}

\begin{abstract}
We give a two variable Rankin-Selberg integral inspired by consideration of Garrett's pullback formula.  For a globally generic cusp form on $\GL_2\times \GSp_4$, the integral represents the product of the $\mathrm{Std}\times \mathrm{Spin}$ and $\mathbf{1} \times \mathrm{Std}$ $L$-functions.  We prove a result concerning an Archimedean principal series representation in order to verify a case of Jiang's first-term identity relating certain non-Siegel Eisenstein series on symplectic groups.  Using it, we obtain a new proof of a known result concerning possible poles of these $L$-functions.
\end{abstract}

\maketitle

\section{Introduction}

Bump, Friedberg, and Ginzburg \cite{bfg2} give three multivariate integrals on $\GSp_4$, $\GSp_6$, and $\GSp_8$ that each represent the product of the Standard and Spin $L$-functions on the group.  There are also works of a similar flavor by Bump-Friedberg \cite{bumpFriedberg} on $\GL_n$ and by Ginzburg-Hundley \cite{ginzh} on certain orthogonal groups.  We produce a similar construction for the product $\GL_2 \times \GSp_4$.  Multivariate integral representations on products of groups are scarce in the literature: there is the integral of Hundley and Shen \cite{hs}, which represents the product of two $\GL_2$-twisted Spin $L$-functions on $\GSp_4$, as well as a very general construction of Jiang and Zhang \cite{jz} that also considers varying $\GL_n$-twists of a fixed representation.

For a precise definition of the main integral and relevant groups, embeddings, and Eisenstein series, refer to Section \ref{subsec:global} below.  Briefly, the integral is as follows.  Define an embedding of $G=\GL_2 \boxtimes \GSp_4$ into $\GSp_6$ (where $\boxtimes$ indicates a matching similitude condition).  Let $V$ be the symplectic space of rank 6 upon which $\GSp_6$ acts.  The normalized Eisenstein series $E^*(g,s,w)$ is a function of $g \in \GSp_6$ as well as two complex parameters $s,w$.  It is associated to a non-maximal parabolic stabilizing a partial flag consisting of a two-dimensional isotropic subspace inside a three-dimensional isotropic subspace of $V$.  If $\pi$ is a generic cuspidal automorphic representation of $G$ with trivial central character, and $\phi$ is a cusp form in the space of $\pi$, we define
\[I^*(\phi,s,w) = \int_{G(\mathbf{Q})Z(\mathbf{A}) \backslash G(\mathbf{A})} E^*(g,s,w) \phi(g)dg.\]
The following is a rough statement of our main result, which is Theorem \ref{thm:mainthm} below.
\begin{thm}\label{thmA:intro} We have
\[I^*(\phi,s,w) {=}_S L(s,\pi,\mathrm{Std}\times \mathrm{Spin})L(w,\pi_2,\mathrm{Std}),\]
where $=_S$ indicates equality away from ramified places.
\end{thm}

\subsection{Heuristic analysis of $I^*(\phi,s,w)$} The integral $I^*(g,s,w)$ and its relation to the $L$-functions $L(s,\pi,\mathrm{Std}\times \mathrm{Spin})L(w,\pi_2,\mathrm{Std})$ can be heuristically motivated by appealing to Garrett's pullback formula \cite{garrett,garrett2}, as follows.  Consider the embedding
\begin{equation} \label{eqn:introsiegel} \underbrace{\GSp_2}_{E(h,s)} \boxtimes \underbrace{\GSp_2}_{\phi_1} \boxtimes \underbrace{\GSp_4}_{\phi_2} \inj \underbrace{\GSp_8,}_{E_{\mathrm{Siegel}}(g,w)}\end{equation}
where $\boxtimes$ denotes the subgroup of the product where the elements share the same similitude, $E(h,s)$ is the usual real analytic Eisenstein series on $\GSp_2=\GL_2$, $E_{\mathrm{Siegel}}(g,w)$ is the usual Siegel Eisenstein series on $\GSp_8$, and $\phi_1 \in \pi_1$ and $\phi_2 \in \pi_2$ are cusp forms in generic cuspidal automorphic representations on $\GSp_2$ and $\GSp_4$, respectively.  Consider the integral of the pullback of the Siegel Eisenstein series $E_{\mathrm{Siegel}}(g,w)$ against the three functions on the embedded groups.  

Due to the presence of the Eisenstein series $E(h,s)$, such an integral does not converge.  However, we may heuristically connect (\ref{eqn:introsiegel}) to a product of $L$-functions as follows.
\begin{enumerate}
	\item The ``doubling'' case of the pullback formula for
	\begin{equation} \label{eqn:introdoub} \GSp_4 \boxtimes \underbrace{\GSp_4}_{\phi_2} \inj \underbrace{\GSp_8}_{E_{\mathrm{Siegel}}(g,w)}\end{equation}
	indicates that pulling back $E_{\mathrm{Siegel}}$ to the product group, then integrating against $\phi_2$ only on the second factor will yield $L(\pi_2,\mathrm{Std},w)$ multiplied by an automorphic form $\widetilde{\phi}_2$ on the first $\GSp_4$ factor.  Furthermore, the automorphic form $\widetilde{\phi}_2$ is essentially the cusp form $\phi_2$, now considered on the first $\GSp_4$ factor.
	\item We embed $\underbrace{\GSp_2}_{E(h,s)} \boxtimes \underbrace{\GSp_2}_{\phi_1}$ into the first $\GSp_4$ factor in (\ref{eqn:introdoub}) and consider its integral against the restriction of $\widetilde{\phi}_2$.  This is the integral representation of Novodvorsky \cite{novodvorsky} for $L(\pi_1 \times \pi_2, \mathrm{Std} \times \mathrm{Spin},s)$.
\end{enumerate}
This suggests that when the complex variables are suitably normalized and convergence issues are ignored, the integral suggested by (\ref{eqn:introsiegel}) should yield the product $L(\pi_2,\mathrm{Std},w)L(\pi_1 \times \pi_2, \mathrm{Std} \times \mathrm{Spin},s)$.

Moreover, we may apply the pullback formula to (\ref{eqn:introsiegel}) in a different way to arrive at the integral $I^*(\phi,s,w)$.  To see this, consider the embedding
\begin{equation}\label{newPullback}\underbrace{\GSp_2}_{E(h,s)} \boxtimes \GSp_6 \rightarrow \underbrace{\GSp_8}_{E_{\mathrm{Siegel}}(g,w)}.\end{equation}
Ignoring issues of convergence and cuspidality, Garrett's pullback formula suggests that if one were to integrate out the Eisenstein series $E(h,s)$ in (\ref{newPullback}) to obtain an automorphic form on $\GSp_6$, this divergent integral would be replaced by the Eisenstein series on $\GSp_6$ corresponding to the parabolic of shape
\[ \( \begin{array}{ccc} m_1 & * & *\\  & m_2 & *\\ & & *\end{array}\)\]
in $2 \times 2$ block form and with (non-cuspidal) data given by the $\GSp_2$ Eisenstein series $E(h,s)$ on the block corresponding to $m_2$.  In other words, what should result is the two-variable Eisenstein series $E^*(g,s,w)$ used in $I^*(\phi,s,w)$.  We should then pull back to an embedded
\[\underbrace{\GSp_2}_{\phi_1} \boxtimes \underbrace{\GSp_4}_{\phi_2} \inj \GSp_6\]
and integrate.

Thus the integral suggested by (\ref{eqn:introsiegel}) is connected to both $L(\pi_2,\mathrm{Std},w)L(\pi_1 \times \pi_2, \mathrm{Std} \times \mathrm{Spin},s)$ and the integral $I^*(\phi,s,w)$.  In Theorem \ref{thmA:intro}, we directly relate $I^*(\phi,s,w)$ to the product of these two $L$-functions.

\subsection{A first term identity and application to the poles}

Under an assumption concerning Archimedean principal series, Jiang \cite{jiang2} proved a general first term identity relating the special values of the different Eisenstein series on symplectic groups.  This was inspired by the Siegel-Weil formula and subsequent work of Kudla and Rallis, both of which concern the Siegel Eisenstein series.  Jiang's thesis \cite{jiang} intriguingly applied a case of this first-term identity that related two \emph{non-Siegel} parabolic subgroups.  The appearance of the two-variable Eisenstein series $E^*(g,s,w)$ in $I^*(g,s,w)$ directly motivates the consideration of another first term identity between two non-Siegel parabolic subgroups, which is given in Proposition \ref{FTQ} below.  We give a simple but imprecise formulation here.
\begin{prop} \label{prop:introFT}
Let $Q$ be the Klingen parabolic stabilizing an isotropic line in the symplectic space $V$ and let $R$ be the parabolic subgroup stabilizing a two-dimensional isotropic subspace of $V$.  We write $E_Q^*(g,s)$ and $E_R^*(g,w)$ for the normalized Eisenstein series on $\mathrm{Sp}_6(\mathbf{A})$ associated to these respective parabolics with degenerate data.  (See Section \ref{subsec:res} for precise definitions.)  Then for any choice of data in $E_R^*(g,w)$, there is a corresponding choice of data for $E_Q^*(g,s)$ so that we have an equality
\[\mathrm{Res}_{w=2} E_R^*(g,w) = E_{Q}^*(g,s)|_{s=3/2}\]
for all $g \in \Sp_6(\A)$.
\end{prop}
We emphasize that the preceding identity is a result of Jiang \cite[Proposition 3.1]{jiang2} assuming only an Archimedean hypothesis, which we verify in Proposition \ref{prop:assume} below, following a method also due to Jiang \cite[\S 4]{jiang}.  This identity can then be applied to the aforementioned integral representation to analyze when poles of the $L$-functions $L(\pi_2,\mathrm{Std},w)$ and $L(\pi_1 \times \pi_2, \mathrm{Std} \times \mathrm{Spin},s)$ can simultaneously exist.  This analysis is similar to the argument of Bump, Friedberg, and Ginzburg \cite{bfg2}, who determined whether poles can simultaneously exist for the Standard and Spin $L$-functions on $\GSp_4$ and $\GSp_6$ using the Siegel-Weil formula of Ikeda \cite{ikeda} and Kudla-Rallis-Soudry \cite{kr,krs}.  It appears that aside from Jiang's thesis \cite{jiang}, Proposition \ref{propB:intro} below is the only other application of an identity relating values of non-Siegel parabolics in the literature.

We also carry out a nonvanishing computation at the Archimedean place and a ramified calculation for our construction.  The analysis is intricate due to the presence of an Eisenstein series associated to a non-maximal parabolic.  The following is an imprecise version of Proposition \ref{prop:ramcalc}.
\begin{prop} \label{prop:ramcalcIntro}
For any finite place $v$, the data in the local version of the unfolded integral may be chosen in order to make the integral constant and non-zero.  If $v$ is Archimedean, the local integral has meromorphic continuation in $s$ and $w$.  Moreover, if $s_0,w_0 \in \mathbf{C}$ are given, it is possible to choose the data so that the local integral is nonvanishing at the point $(s_0,w_0)$.
\end{prop}

Write the automorphic representation $\pi$ as the tensor product $\pi_1 \times \pi_2$ on $\GL_2 \times \GSp_4$.  Piatetski-Shapiro and Rallis \cite{psr} have shown that the Standard $L$-function of $\pi_2$ has possible simple poles only at $w \in \set{-1,0,1,2}$.  (By the functional equation, it is enough to study the cases $w \in \set{1,2}$.)  Soudry \cite{soudry} has completely characterized the existence of the pole $w=2$, while Kudla, Rallis, and Soudry \cite{krs} have done the same for the pole $w=1$.  On the other hand, it follows from the integral of Novodvorsky \cite{novodvorsky} that $L(s,\pi,\mathrm{Std} \times \mathrm{Spin})$ can have only a simple pole at $s \in \set{0,1}$.  Piatetski-Shapiro and Soudry \cite[Theorem 1.3]{pss} have given a necessary and sufficient criterion for the existence of such a pole in terms of the theta lifting from split $\mathrm{GO}(4)$.   We may deduce the following consequence from Proposition \ref{prop:introFT}, the ramified calculation Proposition \ref{prop:ramcalcIntro}, and the integral representation.
\begin{prop}\label{propB:intro}
Suppose that $\pi$ is a generic cuspidal automorphic represention of $G$, and fix a finite set $S$ of places including the Archimedean place and every place of ramification of $\pi$.  If the partial $L$-function $L^S(w,\pi_2,\mathrm{Std})$ has a pole at $w=2$, then $L^S(s,\pi,\mathrm{Std} \times \mathrm{Spin})$ cannot have a pole at $s=1$.
\end{prop}
By comparing the conditions in \cite{krs} and \cite{pss} (both of which characterize the poles in terms of liftings from split $\mathrm{GO}(4)$), one sees that $L(w, \pi_2, \mathrm{Std})$ can have a pole at $w=1$ while $L(s,\pi,\mathrm{Std} \times \mathrm{Spin})$ has a pole at $s=1$ (for a suitable choice of $\pi_1$), so the above result cannot be extended to $(w,s) = (1,1)$.

Proposition \ref{propB:intro} is strictly weaker than one should expect, for the following reason.  Soudry \cite[Theorem A]{soudry} shows that if $L(w,\pi_2,\mathrm{Std})$ has a pole at $w=2$, then $\pi_2$ is a CAP form.  It is a conjecture that such forms are never generic, which can perhaps be deduced from work on the Arthur conjectures in this case, though a more direct argument is possible.  We check below in Proposition \ref{prop:nopole} that the non-existence of a pole of $L(w,\pi_2,\mathrm{Std})$ at $w=2$ follows from the two-variable integral representation of Bump, Friedberg, and Ginzberg \cite{bfg2}.

We now give the layout of the paper.  We unfold the integral $I^*(\phi,s,w)$ in Section \ref{sec:unfolding} and complete the unramified calculation in Section \ref{sec:unramified}.  The first term identity, ramified calculations, and analysis of the poles are given in Section \ref{sec:poles}.

{\bf Acknowledgements}: We thank Solomon Friedberg, Michael Harris, and Christopher Skinner for helpful discussions.  We would also like to thank the anonymous referee for suggestions that have improved the writing of this paper.

\section{Unfolding} \label{sec:unfolding}

We precisely define the global integral $I^*(\phi,s,w)$ of the introduction in Section \ref{subsec:global} below.  We then unfold it to a factorizable form in Section \ref{subsec:unfolding}.  We perform an inner unipotent integration of the local factors in Section \ref{subsec:unipotent}, which will prepare us to conduct an unramified computation in Section \ref{sec:unramified}.

\subsection{The global integral} \label{subsec:global}

Let $W_{2n}$ be a $2n$-dimensional vector space over $\mathbf{Q}$.  We fix an ordered basis $\set{e_1,\dots,e_n,f_n,\dots,f_1}$ for $W_{2n}$, which will be used when writing matrices of linear operators on $W_{2n}$.  We define a symplectic form $\langle \; ,\; \rangle$ on $W_{2n}$ by
\begin{equation}\label{eqn:defsymp} \langle e_i,f_j \rangle = \delta_{ij}\textrm{ and }\langle e_i,e_j\rangle  = \langle f_i,f_j\rangle = 0,\end{equation}
where $\delta_{ij}$ is the Kronecker delta function.  Define an algebraic group $\GSp_{2n}$ over $\mathbf{Q}$ by
\begin{align*}
  \GSp_{2n}(R) = & \left\lbrace M \in \GL_{2n}(R): \textrm{there exists }\mu(M) \in R^\times\right.\\
	&\left.\textrm{with }\langle Mv,Mw \rangle = \mu(M)\langle v,w\rangle \textrm{ for all }v,w \in W_{2n}\right\rbrace
\end{align*}
for $\mathbf{Q}$-algebras $R$.  This also defines the similitude homomorphism $\mu: \GSp_{2n} \rightarrow \mathbf{G}_m$.  We regard $\GSp_{2n}$ as acting on the right of $W_{2n}$.

For $m,n \in \mathbf{Z}_{>0}$ we define the algebraic group $\GSp_{2m} \boxtimes \GSp_{2n}$ over $\mathbf{Q}$ by
\[(\GSp_{2m} \boxtimes \GSp_{2n})(R) = \set{M \in (\GSp_{2m} \times \GSp_{2n})(R): \mu_1(M) = \mu_2(M)},\]
where $\mu_1: \GSp_{2m} \rightarrow \mathbf{G}_m$ and $\mu_2: \GSp_{2n}\rightarrow \mathbf{G}_m$ are the similitude maps.

Now let $V = W_6$.  We define a decomposition $V = V_1 \oplus V_2$ by $V_1=\gen{e_1,f_1}$ and $V_2 =\gen{e_2,e_3,f_3,f_2}$, where $\gen{\cdot}$ denotes the span.  Using (\ref{eqn:defsymp}) to define the forms $\langle \cdot , \cdot \rangle_i$ on $V_i$, there are identifications $V_1 \cong W_2$ and $V_2 \cong W_4$.  We obtain an embedding $\GSp_2 \boxtimes \GSp_4 \rightarrow \GSp_6$.  When considering this subgroup, we will identity $\GSp_2$ with $\GL_2$ in what follows.  We also write $B, B_1$ and $B_2$ for the upper-triangular Borel subgroups of $\GSp_6$, $\GL_2$, and $\GSp_4$, respectively, and $U_B,U_{B_1}$, and $U_{B_2}$ for the respective unipotent radicals.

We define the parabolic subgroup $P \subseteq \GSp_6$ as the stabilizer of the flag $\gen{f_1,f_2} \subseteq \gen{f_1,f_2,f_3}$.  With respect to our ordered basis,
\begin{equation} \label{eqn:defp} P = \set{ \(\begin{array}{cccc} m_1 & * & * & * \\ & m_2 & * & * \\  & & \mu m_2^{-1} & *\\ & & & \mu m_1^*\end{array}\) \in \GSp_6: m_1 \in \GL_2, m_2 \in\GL_1, \mu \in \mathbf{G}_m}.\end{equation}
The notation $m_1^*$ means that the matrix is determined by $m_1$ and the symplectic condition. The Levi $M$ of $P$ is isomorphic to $\GL_2 \times \GL_1 \times \mathbf{G}_m$, where as in (\ref{eqn:defp}) we write the last factor as $\mathbf{G}_m$ to emphasize that it comes from the similitude.  Let $\mathbf{A}$ denote the adeles of $\mathbf{Q}$.  We define a character $\chi_{w,s}: M(\mathbf{A}) \rightarrow \mathbf{C}^\times$ by $\chi_{w,s}(m_1,m_2,\mu) = |\det m_1|_{\mathbf{A}}^{w}|m_2|_{\mathbf{A}}^{2s}|\mu|_{\mathbf{A}}^{-s-w}$.  Let $f_{w,s} = \prod_v f_{v,w,s}$ denote a factorizable section of $\ind_{P(\mathbf{A})}^{\GSp_6(\mathbf{A})} \chi_{w,s}$, where here we mean the non-normalized induction.  We require that at all but finitely many primes $p$, $f_{p,w,s}$ is unramified, by which we mean that $f_{p,w,s}$ is fixed by $G(\mathbf{Z}_p)$ and normalized so that $f_{p,w,s}(1)=1$.  We then define the function
\[E(g,w,s) = \sum_{\gamma \in P(\mathbf{Q}) \backslash \GSp_6(\mathbf{Q})} f_{w,s}(\gamma g)\]
on $\GSp_6(\mathbf{A})$.  By \cite[Proposition II.1.5]{mw}, this sum converges if both $\mathrm{Re}(w-2s)$ and $\mathrm{Re}(s)$ are sufficiently large, and the series can be meromorphically continued in $w$ and $s$ by work of Langlands.

For a finite set of places $S$, we define a normalized Eisenstein series
\[E_S^*(g,s,w) = \zeta^S(2s)\zeta^S(w-1)\zeta^S(2w-4)\zeta^S(2s+w-2)\zeta^S(w-2s) E(g,s,w).\]
The normalized Eisenstein series satisfies a functional equation relating $w$ to $5-w$ and $s$ to $1-s$.

Let $\pi_1$ and $\pi_2$ be cuspidal automorphic representations with trivial central character on $G_1=\GL_2$ and $G_2=\GSp_4$, respectively.  For $i \in \set{1,2}$, we write $\phi_i \in \pi_i$ for an automorphic form that, when viewed as a vector in a fixed factorization $\pi_i = \bigotimes_v \pi_{i,v}$ of $\pi_i$ into representations of each $G_i(\mathbf{Q}_v)$, corresponds to a pure tensor.  Let $Z$ be the common center of $\GSp_6$ and $\GL_2 \boxtimes \GSp_4$.  Define the non-normalized Rankin-Selberg integral
\[I(\phi_1,\phi_2, s,w) = \int_{(\GL_2\boxtimes \GSp_4)(\mathbf{Q})Z(\mathbf{A}) \backslash (\GL_2\boxtimes \GSp_4)(\mathbf{A})} E(g,w,s)\phi_1(g_1)\phi_2(g_2)dg,\]
where $g_1$ and $g_2$ denote the respective projections to $\GL_2$ and $\GSp_4$.  Then our main result is the following theorem.

\begin{thm} \label{thm:mainthm}
Define the normalized Rankin-Selberg integral
\[I^*_S(\phi_1,\phi_2, s,w) = \int_{(\GL_2\boxtimes \GSp_4)(\mathbf{Q})Z(\mathbf{A}) \backslash (\GL_2\boxtimes \GSp_4)(\mathbf{A})} E^*_S(g,w,s)\phi_1(g_1)\phi_2(g_2)dg.\]
Then
\begin{equation} \label{eqn:mainthm} I^*_S(\phi_1,\phi_2, s,w) =_S L(w-2,\pi_1\times \pi_2,\mathbf{1}\times \mathrm{Std})L(s,\pi_1\times \pi_2,\mathrm{Std}\times \mathrm{Spin}),\end{equation}
where the notation $=_S$ means that both sides factorize into a product of functions at each place of $\mathbf{Q}$ that are equal term-by-term away from the finite set $S$ of places that are either infinite or at which $\phi_1$, $\phi_2$, or $f_{s,w}$ is ramified.
\end{thm}

\subsection{Unfolding calculation} \label{subsec:unfolding}
We prove the following result, which shows that the global integral on the left-hand side of (\ref{eqn:mainthm}) can be expressed as a product of local integrals involving Whittaker functions.  First, we fix some notation.  Write $H = \GL_2 \boxtimes \GSp_4$.

\subsubsection{Notation} For $y$ in $V = W_6$, denote by $E(y)$ the abelian unipotent group inside $\GSp_6 \supseteq H$ consisting of the maps $v \mapsto v + u\langle v,y \rangle y$ for $u \in \mathbf{G}_a$.  For instance, we have $E(f_1) = U_{B_1}$, where we recall that $U_{B_i}$ denotes the Borel of our fixed embedded copy of $\GSp_{2i}$ for $i \in \set{1,2}$.  We fix an additive character $\psi: \Q\backslash \A \rightarrow \C^\times$ of conductor $1$.  Define the Whittaker coefficient of $\phi_1$ to be
\begin{equation}\label{eqn:whit1} W_{\phi_1}(g) = \int_{E(f_1)(\Q)\backslash E(f_1)(\A)}{\psi^{-1}(\langle e_1, e_1 \cdot n \rangle)\phi_1(n g)\,dn}.\end{equation}
(Recall that $e_1 \cdot n$ denotes the right action of $n \in \GSp_6$ on $V$.) Also let
\begin{equation}\label{eqn:whit2} W_{\phi_2}(g) = \int_{U_{B_2}(\Q)\backslash U_{B_2}(\A)}{\psi^{-1}(\langle e_2 \cdot u,f_3 \rangle + \langle e_3, e_3 \cdot u \rangle)\phi(ug)\,du}\end{equation}
be the Whittaker coefficient associated to $\phi_2$.

Finally, define $E(f_3)'$ to be the subgroup of $H$ consisting of the maps 
\[v \mapsto v - u \langle v, f_1\rangle f_1 + u\langle v, f_3 \rangle f_3\]
for $u$ in $\mathbf{G}_a$ and let $\gamma_5$ denote an element of $\Sp_6(\Q)$ that satisfies $\gen{f_1, f_2} \gamma_5 = \gen{f_1 +f_3, e_1-e_3}$ and $\gen{f_1, f_2, f_3} \gamma_5 = \gen{f_1 + f_3, e_1 -e_3, f_2}.$ (The particular choice of $\gamma_5$ does not matter, although we will choose a specific one below.)
\begin{thm} \label{unfolding} With notation as above,
\begin{equation} \label{eqn:adelicint} I(\phi_1,\phi_2,s,w) = \int_{E(f_3)'(\A)E(f_2)(\A)Z(\A)\backslash H(\A)}{W_{\phi_1}(g_1) W_{\phi_2}(g_2)f(\gamma_5 g,s ,w)\,dg}. \end{equation}
\end{thm}

Before proving the theorem, we will first perform an orbit calculation that will be used to unfold the Eisenstein series.  Recall that $P(\mathbf{Q}) \backslash G(\mathbf{Q})$ may be identified with the space of partial isotropic flags of the form $F_2 \subseteq F_3$, where $F_i$ is an $i$-dimensional isotropic subspace of $V$.

\begin{lemma}\label{orbitCalc}
The double coset space $P(\mathbf{Q}) \backslash G(\mathbf{Q})/H(\mathbf{Q})$ consists of five elements represented by the flags
\begin{enumerate}
	\item $\gen{f_2,f_3} \subseteq \gen{f_1,f_2,f_3}$,
	\item $\gen{f_1,f_2} \subseteq \gen{f_1,f_2,f_3}$,
	\item $\gen{f_1+f_2,f_3} \subseteq \gen{f_1,f_2,f_3}$,
	\item $\gen{f_1+f_2,f_3} \subseteq \gen{f_1+f_2,e_1-e_2,f_3}$, and
	\item $\gen{f_1+f_2,e_1-e_2} \subseteq \gen{f_1+f_2,e_1-e_2,f_3}$.
\end{enumerate}

\end{lemma}

\begin{proof}

We need to classify the $H(\mathbf{Q})$ orbits of the flags parametrized by $P(\mathbf{Q}) \backslash G(\mathbf{Q})$.  We first consider orbits for the 2-dimensional space $F_2$, and then classify extensions of these to the flag $F_2 \subseteq F_3$.

Recall that $\GL_2(\mathbf{Q})$ acts on $\gen{e_1,f_1}$ while $\GSp_4(\mathbf{Q})$ acts on $\gen{e_2,e_3,f_2,f_3}$.  This corresponds to the fixed splitting $V = V_1 \oplus V_2$.  We first consider case (a), which is where $F_2 \subseteq V_2$.  In this case, $F_2$ can be moved to $\gen{f_2,f_3}$ by $\GSp_4(\mathbf{Q})$, and conversely, the action of $H$ permutes such spaces, so this is the only needed representative.  For the remaining orbits, we may assume that $F_2$ is not contained in $V_2$.

We further subdivide into case (b), which is when $F_2$ contains a vector in $V_1$, or case (c), where $F_2$ does not contain a vector of $V_1$ and is not contained in $V_2$.  In case (b), we pick $v_1 \in F_2$, and in case (c), we pick any $v_1 \in F_2$ whose $V_1$ component is non-zero.  In either case, we move by $\GL_2(\mathbf{Q})$ so that the $V_1$ component is $f_1$, choosing an arbitrary action on $V_2$ with the right similitude to obtain an element of $H(\mathbf{Q})$.  In case (c), we further act by an element of $H$ whose $\GSp_4$ part moves the $V_2$ component of $v_1$ to $f_2$ and whose $\GL_2$ component fixes $f_1$ and scales $e_1$ suitably to meet the similitude condition.  In particular, $v_1 = f_1$ in case (b) and $v_1= f_1+f_2$ in case (c).  In either case, we pick a second basis element $v_2 \in F_2$ so that when $v_2$ is written in our chosen basis, the $f_1$ coefficient is 0.  Due to $F_2$ being isotropic, $v_2 \notin V_1$.  In case (b), the fact $\langle v_1,v_2 \rangle = 0$ forces $v_2 \in V_2$, so we can act by $H$ to move $v_2$ to $f_2$ while fixing $v_1$.  This shows that the flags in case (b) form a single orbit represented by $\gen{f_1,f_2}$.

In case (c), we further subdivide into the case (c1) that $v_2 \in V_2$ or (c2) that $v_2 \notin V_2$.  In case (c1), $v_2$ and the $V_2$ component of $v_1$ generate a maximal isotropic subspace of $V_2$, so we may move $v_2$ to $f_3$ while fixing $f_1$ and $f_2$.  In particular, the $F_2$ that contain no vector in $V_1$ but do contain a vector in $V_2$ form a single orbit represented by $\gen{f_1+f_2,f_3}$.

In case (c2), first note that the $V_2$ component of $v_2$ is not a multiple of $f_2$, else we would be in case (b).  Moreover, writing $v_i = v_i^{(2)} + v_i^{(4)} \in V_1 \oplus V_2 = V$ for $i \in \set{1,2}$, we have $\langle v_1^{(2)}, v_2^{(2)} \rangle = -\langle v_1^{(4)}, v_2^{(4)} \rangle$.  We may first move by an element of $H$ that fixes the $f_i$ and scales the $e_i$ by a constant so that $v_2^{(2)}=e_1$ and $v_2^{(4)} = -e_2 +be_3+cf_2+df_3$.  We may then move by an element of $\Sp_4(\mathbf{Q}) \subseteq H(\mathbf{Q})$ that fixes $v_1^{(4)}$ and sends $-v_2^{(4)}$ to $e_2$ since it acts transitively on symplectic bases (and $v_1^{(4)}$ and $-v_2^{(4)}$ may be extended to such a basis).  So all $F_2$ containing no vectors in $V_1$ or $V_2$ form a single orbit represented by $\gen{f_1+f_2,e_1-e_2}$.

We now consider possibilities for the maximal isotropic subspace $F_3$ containing $F_2$ in each case.  For case (a), since $F_3$ is isotropic, a vector $v_3 \in F_3 \setminus F_2$ must not lie in $V_2$, and by subtracting elements of $F_2$, we can assume $v_3 \in V_1$.  We may then move the flag to $\gen{f_2,f_3} \subseteq \gen{f_1,f_2,f_3}$, which is representative (1) above.  For case (b), take any $v_3 \in F_3 \setminus F_2$ and subtract a multiple of $v_1=f_1$ so that (since $F_3$ is isotropic) we have $v_3 \in V_2$.  We can then use an element of $H(\mathbf{Q})$ to move $v_3$ to $f_3$ while fixing $v_1$ and $v_2$, giving the flag $\gen{f_1,f_2} \subseteq \gen{f_1,f_2,f_3}$, which is representative (2).

In case (c1), we subdivide further into the cases (c1i) that $F_3$ contains a maximal isotropic subspace of $V_2$ and (c2ii) that $F_3$ does not contain a maximal isotropic subspace of $V_2$.  In case (c1i), it is easy to see that $(F_3 \setminus F_2) \cap V_2$ must be nonempty, and given $v_3$ in this intersection, it may be moved by $H(\mathbf{Q})$ to $f_2$ while fixing $v_1$ and $v_2$.  This gives us representative (3) above.  In case (c1ii), let $v_3 \in F_3 \setminus F_2$, and note that we cannot have $v_3 \in V_2$ or $\gen{v_2,v_3}$ would be a maximal isotropic subspace of $V_2$.  Subtracting a multiple of $v_1$ and $v_2$ and multiplying by an element of $H(\mathbf{Q})$ that scales the $e_i$ and fixes the $f_i$, we can arrange using the argument in case (c2) above that in the decomposition $v_3 = v_3^{(2)}+v_3^{(4)}$, $v_3^{(2)}=e_1$ and $v_3^{(4)} = -e_2 +be_3+cf_2+df_3$.  As before, we can use an element of $\Sp_4(\mathbf{Q}) \subseteq H(\mathbf{Q})$ to move $-v_3^{(4)}$ to $e_2$ while fixing $v_1^{(4)}$ and $v_2^{(4)}$.  We obtain the representative $\gen{f_1+f_2,f_3} \subseteq \gen{f_1+f_2,e_1-e_2,f_3}$ in this case, which is (4) above.

We finally consider case (c2).  Starting with any $v_3 \in F_3 \setminus F_2$, we may subtract multiples of $v_1$ and $v_2$ to get $v_3 \in V_2$.  Using the notation above, we observe that $\langle v_1^{(4)},v_3^{(4)} \rangle = \langle -v_2^{(4)},v_3^{(4)} \rangle = 0$, so $v_1^{(4)},v_2^{(4)},$ and $v_3^{(4)}$ may be extended to a symplectic basis.  We move by an element of $\Sp_4(\mathbf{Q}) \subseteq H(\mathbf{Q})$ as before to obtain representative (5) above.
\end{proof}

We remark that the qualitative descriptions of the cases given in the proof are all invariant under the action of $H(\mathbf{Q})$ and thus provide natural characterizations of the five orbits.

\begin{proof}[Proof of Theorem \ref{unfolding}] For $i = 1,2,3,4$, denote by $F_2^i \subseteq F_3^i$ the corresponding $(2,3)$ flag from Lemma \ref{orbitCalc}, i.e.\ the $i^\textrm{th}$ flag listed in this lemma.  For $i = 5$, set $F_2^5 = \gen{f_1 + f_3, e_1 - e_3}$, $F_3^5 = \gen{ f_1 + f_3, e_1 - e_3, f_2} = F_3$.  Clearly this flag is in the same $H$ orbit as the $5$th flag of Lemma \ref{orbitCalc}.

Suppose $\gamma_1, \ldots , \gamma_5$ are elements of $\GSp_6(\Q)$ such that $\gen{f_1,f_2}\gamma_i = F_2^i$ and $\gen{f_1,f_2,f_3} \gamma_i = F_3^i$. Denote by $\mathrm{Stab}_i$ the stabilizer inside $H$ of the flag $F_2^i \subseteq F_3^i$.  Then the global integral is a sum $\sum_{i =1}^{5}{I_i(\phi_1, \phi_2,s,w)}$, where
\[I_i(\phi_1,\phi_2,s,w) = \int_{\mathrm{Stab}_i(\Q)Z(\A)\backslash H(\A)}{\phi_1(g_1)\phi_2(g_2)f(\gamma_i g,s,w)\,dg}.\]
We claim the integrals $I_i$ vanish for $1 \leq i \leq 4$.  Indeed, the unipotent group $U_{B_1} \times \mathbf{1}_4 \subseteq H$ is contained inside $\mathrm{Stab}_i$ for $i = 1,2,3$, so these integrals vanish by the cuspidality of $\phi_1$.  Now let $i=4$.  Denote by $U(f_3)$ the unipotent radical of the parabolic subgroup of $\GSp_4$ that stabilizes the line spanned by $f_3$.  Then $\mathbf{1}_2 \times U(f_3) \subseteq H$ is contained inside $\mathrm{Stab}_4$.  Moreover, the section $f(\gamma_4 g,s,w)$ is left-invariant by $U(f_3)$ since the conjugate $\gamma_4 U(f_3) \gamma_4^{-1}$ is contained in the parabolic $P$.  Therefore the integral $I_4(\phi_1,\phi_2,s,w)$ vanishes by the cuspidality of $\phi_2$.

We now unfold the integral $I_5$.  First, we compute $\mathrm{Stab}_5$.
\begin{lemma} \label{lem:gamma5stab} An element $(g_1, g_2) \in (\GL_2 \boxtimes \GSp_4)(\mathbf{Q})$ is in $\mathrm{Stab}_5(\mathbf{Q})$ if and only if they have the form
\[g_1 = \mm{a}{-b}{-c}{d} \textrm{ and } g_2 = \left(\begin{array}{cccc} * & & & *\\ & a & b & \\ & c&d & \\ & & & *\end{array}\right),\]
where the ordered basis used to write the second matrix is $\set{e_2, e_3, f_3, f_2}$.
\end{lemma}
\begin{proof} Suppose $g = (g_1, g_2)$ is in the stabilizer.  First note that the line spanned by $f_2$ must be taken to itself.  Indeed, since $f_2 \in F_3^5$, $f_2 \mapsto A f_2 + B(f_1 + f_3) + C(e_1 -e_3)$ for some $A,B,C \in\mathbf{Q}$.  But since $g_2 \in \GSp_4(\mathbf{Q})$, $B = C = 0$, since $e_1$ and $f_1$ cannot occur with nonzero coefficients.  Thus $g_2$ does indeed stabilize the line spanned by $f_2$.

Next, we see that $g_2$ takes $\gen{e_3, f_3}$ to itself, since $g_2$ stabilizes $F_2^5$.  Suppose the action of $g_2$ on this space is given by the matrix $\mm{a}{b}{c}{d}$, as in the statement of the lemma.  Then by a simple computation one sees that $(g_1,g_2)$ is in the stabilizer if and only if $g_1 = \mm{a}{-b}{-c}{d}$, as claimed. \end{proof}

Then the (open orbit) integral unfolds as
\begin{align*} I_5(\phi_1, \phi_2, s,w) &= \int_{\mathrm{Stab}_5(\Q)Z(\A)\backslash H(\A)}{\phi_1(g_1) \phi_2(g_2)f(\gamma_5 g,s ,w)\,dg} \\ &= \int_{\mathrm{Stab}_5(\Q)E(f_2)(\A)Z(\A)\backslash H(\A)}{\phi_1(g_1) \phi_{2,0}(g_2)f(\gamma_5 g,s ,w)\,dg} \end{align*}
where
\[\phi_{2,0}(g_2) = \int_{E(f_2)(\Q)\backslash E(f_2)(\A)}{\phi(ng_2)\,dn}.\]
Denote by $U(f_2)$ the unipotent radical of the parabolic subgroup of $\GSp_4$ that stabilizes the line spanned by $f_2$.  This is a Heisenberg group, with center the group $E(f_2)$.  Define the character $\chi_1:U(f_2)(\Q)\backslash U(f_2)(\A) \rightarrow \mathbf{C}^\times$ by $\chi_1(n) = \psi(\langle e_2 \cdot n,f_3\rangle)$.  Fourier expanding $\phi_{2,0}$ along $U(f_2)/E(f_2)$, one obtains
\[I_5(\phi_1, \phi_2, s,w) = \int_{B_\triangle(\Q)E(f_2)(\A)Z(\A)\backslash H(\A)}{\phi_1(g_1) \phi_{2,\chi_1}(g_2)f(\gamma_5 g,s ,w)\,dg},\]
where
\[\phi_{2,\chi_1}(g_2) = \int_{U(f_2)(\Q)\backslash U(f_2)(\A)}{\chi_1^{-1}(n)\phi_2(ng_2)\,dn}\]
and $B_\triangle \subseteq \mathrm{Stab}_5$ is the group of matrices of the form
\[(g_1,g_2) = (\mm{t}{-b}{}{t'},\left(\begin{array}{cccc} t & & & \\ & t & b & \\ & &t' & \\ & & & t'\end{array}\right)).\]
We let $T' \subseteq B_\triangle$ denote the diagonal maximal torus, which is embedded in $H$ as the matrices of the form $\mathrm{diag}(t,t,t,t',t',t')$.

We have $B_\triangle = T' E(f_3)'$, where $E(f_3)'\subseteq \mathrm{Stab}_5$ is defined as above to consist of the maps 
\[v \mapsto v - u \langle v, f_1\rangle f_1 + u\langle v, f_3 \rangle f_3\]
for $u$ in $\mathbf{G}_a$.  Applying the Whittaker expansion $\phi_1(g_1) = \sum_{\gamma \in T'(\Q)}{W_{\phi_1}(\gamma g_1)}$ and integrating over $E(f_3)'$, we obtain
\begin{align*} I_5(\phi_1, \phi_2, s,w) &= \int_{T'(\Q)E(f_3)'(\Q)E(f_2)(\A)Z(\A)\backslash H(\A)}{\phi_1(g_1) \phi_{2,\chi_1}(g_2)f(\gamma_5 g,s ,w)\,dg} \\ &= \int_{E(f_3)'(\Q)E(f_2)(\A)Z(\A)\backslash H(\A)}{W_{\phi_1}(g_1) \phi_{2,\chi_1}(g_2)f(\gamma_5 g,s ,w)\,dg} \\ &= \int_{E(f_3)'(\A)E(f_2)(\A)Z(\A)\backslash H(\A)}{W_{\phi_1}(g_1) W_{\phi_2}(g_2)f(\gamma_5 g,s ,w)\,dg}. \end{align*}
This completes the proof of the theorem.
\end{proof}

The following lemma specifies a suitable element $\gamma_5$.  Its proof is immediate.
\begin{lemma} \label{lem:gamma5calc} The elements $e_3, -f_1, e_2, f_2, e_1-e_3, f_1 + f_3$ form an ordered symplectic basis.  Consequently, we may choose the element $\gamma_5\in \Sp_6(\Q)$ to be the one that takes
\[e_1 \mapsto e_3, e_2 \mapsto -f_1, e_3 \mapsto e_2, f_3 \mapsto f_2, f_2 \mapsto e_1-e_3,\textrm{ and } f_1 \mapsto f_1 + f_3.\]
\end{lemma}

For a local field $F$ and Whittaker functions $W_1$ on $\GL_2(F)$ and $W_2$ on $\GSp_4(F)$, define the local integral
\begin{equation}\label{eqn:localint} I(W_1, W_2,s,w) = \int_{E(f_3)'(F)E(f_2)(F)Z(F)\backslash H(F)}{W_1(g_1) W_2(g_2)f(\gamma_5 g,s ,w)\,dg}. \end{equation}
By the uniqueness of Whittaker models, the adelic integral (\ref{eqn:adelicint}) factorizes into the product of the local integrals $I(W_1,W_2,s,w)$.
\subsection{Unipotent integration} \label{subsec:unipotent}
We begin the computation of the local integrals $I(W_1,W_2,s,w)$ for unramified data by computing a certain unipotent integral.  In this subsection, $F = \mathbf{Q}_p$ for a rational prime $p$.  We assume that the local components $\pi_{1,p}$ and $\pi_{2,p}$ are unramified.  Since everything is local, we write $f(g,s,w)$ for the section $f_{p,s,w}$.  We also write $K = G(\mathbf{Z}_p)$.

We have assumed that each $\pi_i$ has trivial central character, so we may consider $\pi_1$ and $\pi_2$ to be representations of $\PGL_2$ and $\PGSp_4$, respectively.  Let $c_1(p)$ denote the conjugacy class in the dual group $\SL_2(\C)$ associated to the local spherical representation $\pi_{1,p}$ of $\PGL_2(F)$ and let $c_2(p)$ denote the conjugacy class in the dual group $\Spin_5(\C)$ associated to the local spherical representation $\pi_{2,p}$ of $\PGSp_4(F)$.  Write $\omega^1$ for the fundamental weight of $\SL_2(\C)$ and $\omega^2_1, \omega^2_2$ for the fundamental weights of $\Spin_5(\C)$.  For nonnegative integers $m$ and $n$, define $A_1[m]$ to be trace of $c_1(p)$ on the representation of $\SL_2(\C)$ with heighest weight $m \omega^1$ and define $B_2[m,n]$ to be the trace of $c_2(p)$ on the representation of $\Spin_5(\C)$ with highest weight $m \omega^2_1 + n \omega^2_2$.

The purpose of this section is to prove the following result.
\begin{thm} \label{thm:localintegral} Set $U = |p|^{w-2}$ and $V = |p|^s$.  Then $\zeta(w+2s-2)\zeta(w-1)\zeta(w-2s)I(W_1,W_2,s,w)$ is equal to
\begin{align} &\sum_{\substack{a,b,c \ge 0\\a \le c \le 2a}} U^{c-a} V^c(1+U+\dots+U^{2a-c})\(\sum_{\subalign{0 \le &e \le b\\0 \le &f}} U^{b-e+f}(V^2)^{e+f}\) A_1[2a-c]B_2[b,c] \label{eqn:localintegral}\\
	+&\sum_{\substack{a,b,c \ge 0\\c < a \le b+c}} U^{a-c} V^{2a-c}(1+U+\dots+U^c)\(\sum_{\subalign{0 \le &e \le -a+b+c\\0 \le &f}} U^{-a+b+c-e+f}(V^2)^{e+f}\) A_1[2a-c]B_2[b,c].\nonumber
\end{align}
\end{thm}

Set $P' = \gamma_5^{-1}P\gamma_5$ and let $F'$ be the flag $F_2' = \gen{f_1 + f_3, e_1 - e_3} \subseteq \gen{f_1 +f_3, e_1 - e_3,f_2} = F_3'$. Then $P'$ is the parabolic subgroup of $\GSp_6$ stabilizing $F'$.  Furthermore, set $\chi'_{s,w}(g) = \chi(\gamma_5 g \gamma_5^{-1})$, and $f'(g,s,w) = f(\gamma_5 g, s,w)$, which is a section in $\ind_{P'}^{\GSp_6}(\chi'_{s,w})$.  The next lemma describes the section $f'(g,s,w) \in \ind_{P'}^G(\chi'_{s,w})$ in an invariant form.

We define two functions $\det_2,\det_3 : G(\mathbf{Q}_p) \rightarrow \mathbf{Q}_p^\times$ as follows.  First, write the element $g \in G(\mathbf{Q}_p)$ as a $6 \times 6$ matrix using the ordered basis
\[\set{e_3, -f_1, e_2, f_2, e_1-e_3, f_1 + f_3}\]
via the right action of $\GSp_6$ on $V=W_6$.  Define $|\det_2 g|$ to be the maximum absolute value of the $2\times 2$ minors coming from the last $2$ rows of $g$.  Similarly, define $|\det_3 g|$ to be the maximum absolute value of the $3\times 3$ minors coming from the last $3$ rows of $g$.  
\begin{lemma}  Suppose that $g \in \GSp_6(\mathbf{Q}_p)$ and that $f'(g,s,w)$ is the unique $K$-spherical element of $\ind_{P'}^{\mathrm{GSp}_6}(\chi'_{s,w})$ such that $f'(1,s,w)=1$.  Then $f'(g,s,w) = |\det_3 g|^{-2s}|\det_2 g|^{2s-w} |\mu(g)|^{s+w}$. \end{lemma}
\begin{proof} This has the correct restriction to $P'$, and is right $K$-invariant. \end{proof}

For $x, y,z\in F$, define $u(x,y,z)\in H(F)$ to be the element $(u_1(z),u_2(y,z))$, where the notation is as follows.  We set $u_1(z) = \mm{1}{z}{}{1} \in \GL_2(F)$.  Using the ordered basis $\set{e_2, e_3, f_3,f_2}$, the element $u_2(x,y) \in \GSp_4(F)$ is defined by
\[u_2(x,y) = \left(\begin{array}{cccc} 1&x&y& 0\\ & 1&0 &y\\ & & 1&-x \\ &&&1\end{array}\right).\]

Denote by $T$ the diagonal maximal torus of $H$, which is also a maximal torus of $\GSp_6$.  Let $t \in T$.  With this notation, we set
\begin{equation}\label{f'Def}f'_\psi(t,s,w) = \int_{x,y,z \in F}{\psi(z)\psi(x)f'(u(x,y,z)t,s,w)\,dx\,dy\,dz}.\end{equation}
As before, denote by $B_1$ the upper triangular Borel of $\GL_2$ and by $B_2$ the upper triangular Borel of $\GSp_4$.  Further, set $K_1(t_1) = \delta_{B_1}^{-1/2}(t_1)W_1(t_1)$ and $K_2(t_2) = \delta_{B_2}^{-1/2}(t_2)W_2(t_2)$.  

Then by the Iwasawa decomposition,
\[I(W_1,W_2,s,w) = \int_{(T/Z)(F)}{\delta_{B_1}^{-1/2}(t_1)\delta_{B_2}^{-1/2}(t_2)f'_\psi((t_1,t_2),s,w) K_1(t_1)K_2(t_2)\,dt}.\]

We must compute $f'_\psi$.  We may write a representative for $[t] \in (T/Z)(F)$ in the form $t = \mathrm{diag}(\alpha \beta, \beta^2 \gamma, \beta \gamma, \beta, 1, \alpha^{-1}\beta \gamma)$.  Then the map $ut$ has the action
\begin{align*} f_1 + f_3 &\mapsto \beta(f_1+f_3)+ \beta(\alpha^{-1} \gamma - 1)f_1 - x f_2 \\ e_1 - e_3 &\mapsto \alpha\beta(e_1 -e_3) + \beta(\alpha- \gamma) e_3 + \alpha^{-1}\beta \gamma z f_1 + yf_2 \\ f_2 &\mapsto f_2 \end{align*}
on the vectors generating $F'$.  In other words, computing in the ordered basis $e_3, -f_1, e_2,f_2, e_1 - e_3, f_1 + f_3$, the bottom three rows of the matrix $ut$ are
\[\left(\begin{array}{ccc|ccc} 0&0&0&1&0&0 \\ \beta(\alpha-\gamma) &-\alpha^{-1}\beta \gamma z & 0 &y &\alpha \beta&0 \\ 0&\beta(1-\alpha^{-1}\gamma)&0&-x&0&\beta \end{array}\right).\]
Using row operations, this $3 \times 6$ matrix is right-$K$-equivalent to
\begin{equation}\label{3by6matrix}\left(\begin{array}{ccc|ccc} 0&0&0&1&0&0 \\ \beta\gamma &\alpha^{-1}\beta \gamma z & 0 &y &\alpha \beta&0 \\ 0&\alpha^{-1}\beta\gamma&0&-x&0&\beta \end{array}\right).\end{equation}
By using (\ref{3by6matrix}), we can compute $|\det_2 ut|$ and $|\det_3 ut|$.
\begin{lemma} If $|\gamma| \leq |\alpha|$, $|\det_3 ut| = |\alpha\beta^2| \max \{1, |z_0|\}$ where $z=\alpha^2 \gamma ^{-1} z_0$.  If $|\gamma| \geq |\alpha|$, $|\det_3 ut| = |\alpha^{-1}\beta^2 \gamma^2| \max \{1, |z_0|\}$ where $z = \gamma z_0$. \end{lemma}
\begin{proof} Computing the various minors, one obtains 
\begin{align*} |\mathrm{det}_3\,ut| &= \max \{ |\alpha \beta^2|, |\beta^2 \gamma|, |\alpha^{-1}\beta^2\gamma^2|, |\beta^2\gamma|, |\alpha^{-1}\beta^2\gamma z|\} \\ &= |\beta^2| \max \{|\alpha|,|\gamma|, |\alpha^{-1}\gamma^2|,|\alpha^{-1}\gamma z|\}. \end{align*}
The lemma follows. \end{proof}
\begin{lemma} If $|\gamma| \leq |\alpha|$, 
\[|\mathrm{det}_2\,ut| = |\alpha \beta^2| \{1,|x_0|,|y_0|, |z_0|, |x_0 z_0|\}\]
where $x = \beta x_0$, $y = \alpha \beta y_0$, and $z = \alpha^2 \gamma^{-1}z_0$.  If $|\gamma| \geq |\alpha|$, 
\[|\mathrm{det}_2\,ut| = |\alpha^{-1}\beta^2 \gamma^2| \max\{1, |x_0|, |y_0 + \alpha^{-1}\gamma x_0z_0|, |z_0|, |x_0z_0|\}\]
where $x = \alpha^{-1}\beta \gamma x_0$, $y = \beta \gamma y_0$, and $z = \gamma z_0$. \end{lemma}
\begin{proof} Similar to above, we compute the various minors to obtain
\[|\mathrm{det}_2\,ut| = |\alpha^{-1} \beta \gamma| \max\{ |\beta \gamma|, |\alpha x|, |\alpha \beta|, |y + xz|, |\beta z| ,|\alpha^2\gamma^{-1} x|, |\alpha \gamma^{-1} y|, |\alpha^2\beta \gamma^{-1}| \}.\]
If $|\gamma | \leq |\alpha|$, this is
\[|\alpha^{-1}\beta \gamma| \max \{ |xz|,|\beta z|, |\alpha^2 \gamma^{-1} x|,|\alpha \gamma^{-1} y|, |\alpha^2 \beta \gamma^{-1}|\}.\]
If instead $|\gamma| \geq |\alpha|$, this is
\[|\alpha^{-1}\beta \gamma| \max \{|\beta \gamma|, |\alpha x|, |y+xz|, |\beta z|, |\alpha \gamma^{-1} xz|\}.\]
The lemma follows.
\end{proof}

Assume $|\gamma| \leq |\alpha|$. Combining the above lemmas, and making the variable changes $x = \beta x_0$, $y = \alpha \beta y_0$, and $z = \alpha^2 \gamma^{-1}z_0$, one gets
\begin{align}\label{gammaSmall} &f'_\psi(t,s,w) = |\alpha^3 \beta^2 \gamma^{-1}| |\alpha \beta^2|^{-w} |\beta^2 \gamma|^{w+s} \nonumber \\ & \quad \times \int_{x,y,z}{\psi(\alpha^2 \gamma^{-1} z) \psi(\beta x)\max\{1,|z|\}^{-2s} \max\{1, |x|,|y|,|z|,|xz|\}^{-(w-2s)}\,dxdydz}. \end{align}
If $|\alpha| \leq |\gamma|$, by making the variable change $x = \alpha^{-1}\beta \gamma x_0$, $y = \beta \gamma y_0$, and $z = \gamma z_0$, one gets
\begin{align}\label{alphaSmall} &f'_\psi(t,s,w) = |\alpha^{-1} \beta^2 \gamma^3| |\alpha^{-1}\beta^2 \gamma^2|^{-w}|\beta^2 \gamma|^{w+s} \nonumber \\ & \quad \times \int_{x,y,z}{\psi(\gamma z) \psi(\alpha^{-1}\beta \gamma x) \max\{1,|z|\}^{-2s} \max \{1,|x|, |y+\alpha^{-1}\gamma xz|, |z|, |xz|\}^{-(w-2s)}\,dxdydz}.\end{align}

The proofs of the following two lemmas are straightforward.  Here, $\zeta(\cdot)$ denotes the Euler factor at $p$ of the usual $\zeta$ function.
\begin{lemma}\label{y lemma} The integral
\[\int_{F}{\max\{|c|,|y|\}^{-u}\,dy} = |c|^{1-u} \frac{\zeta(u-1)}{\zeta(u)}.\]
\end{lemma}

\begin{lemma}\label{x lemma} The integral
\[\int_{F}{\psi(ax)\max\{1,|x|\}^{-u}\,dx} = \begin{cases} 0 &\mbox{if } |a| > 1 \\ \frac{1}{\zeta(u)}(1 + |p|^{u-1} + \cdots + |p|^{(u-1)\mathrm{ord}_p(a)}) &\mbox{if } |a| \leq 1. \end{cases}\]
\end{lemma}

Assume $|\gamma| \leq |\alpha|$.  Integrating over $y$ first in (\ref{gammaSmall}) and applying Lemma \ref{y lemma}, we obtain
\begin{align*} f'_\psi(t,s,w) &= |\alpha^3 \beta^2 \gamma^{-1}| |\alpha \beta^2|^{-w} |\beta^2 \gamma|^{s+w} \frac{\zeta(w-2s-1)}{\zeta(w-2s)} \\ & \quad \times \int_{x,z}{\psi(\alpha^2 \gamma^{-1} z) \psi(\beta x)\max\{1,|z|\}^{-2s} \max\{1, |x|,|z|,|xz|\}^{1+2s-w}\,dxdydz}.\end{align*}
The integral in this last expression is
\[\left(\int_{x}{\psi(\beta x)\max\{1,|x|\}^{1+2s-w}\,dx}\right) \left(\int_{z}{\psi(\alpha^2 \gamma^{-1} z) \max\{1,|z|\}^{1-w}\,dz}\right),\]
which can then be evaluated by Lemma \ref{x lemma}.  One obtains, if $|\gamma| \leq |\alpha|$,
\begin{align*} f'_\psi(t,s,w) &= |\beta^2 \gamma|^{s} |\gamma \alpha^{-1}|^{w} |\alpha^3 \beta^2 \gamma^{-1}| (\zeta(w-2s)\zeta(w-1))^{-1} \\ & \quad \times \left(1 + |p|^{w-2} + \cdots + |p|^{(w-2)\ord_p(\alpha^2 \gamma^{-1})}\right)\left(1 + |p|^{w-2s-2} + \cdots + |p|^{(w-2s-2)\ord_p(\beta)}\right). \end{align*}

If $|\alpha| \leq |\gamma|$, then integrating over $y$ first in (\ref{alphaSmall}) and making the variable change $y \mapsto y - \alpha^{-1}\gamma xz$, we get
\begin{align*} f'_\psi(t,s,w) &= |\alpha^{-1}\beta^2 \gamma^3| |\alpha^{-1}\beta^2 \gamma^2|^{-w}|\beta^2 \gamma|^{w+s}\frac{\zeta(w-2s-1)}{\zeta(w-2s)} \\ & \quad \times \int_{x,z}{\psi(\gamma z) \psi(\alpha^{-1}\beta \gamma x) \max\{1,|z|\}^{1-w}\max\{1,|x|\}^{1+2s-w}\,dxdz}. \end{align*}
Hence when $|\alpha| \leq |\gamma|$, for the local integral to be nonvanishing one needs $|\beta| \leq |\alpha \gamma^{-1}|\leq 1$, and one gets
\begin{align*} f'_\psi(t,s,w)&= (\zeta(w-2s) \zeta(w-1))^{-1} |\alpha \gamma^{-1}|^{w}|\beta^2 \gamma|^{s} |\alpha^{-1}\beta^2 \gamma^{3}| \\ &\quad \times\left(1 + |p|^{w-2} + \cdots + |p|^{(w-2)\ord_p(\gamma)}\right)\left(1 + |p|^{w-2s-2} + \cdots + |p|^{(w-2s-2)\ord_p(\alpha^{-1}\beta \gamma)}\right).\end{align*}

Write 
\[I(W_1,W_2,s,w) = I_{|\gamma| \leq |\alpha|}(W_1,W_2,s,w) + I_{|\alpha| < |\gamma|}(W_1,W_2,s,w)\]
where the domain of the first integral in this sum is over the $\alpha,\beta,\gamma$ with $|\gamma|\leq |\alpha|$ and the domain of the second integral in this sum is over those $\alpha,\beta,\gamma$ with $|\alpha|< |\gamma|$.  

We have $\delta_{B_1}(t_1)^{-1/2}\delta_{B_2}(t_2)^{-1/2} = |\alpha \beta^2 \gamma|^{-1}$.  We conclude that 
\begin{align*} \zeta(w-2s) \zeta(w-1) I_{|\gamma| \leq |\alpha|}(W_1,W_2,s,w)
  =&\int_{\alpha,\beta,\gamma} |\beta^2 \gamma|^{s}|\gamma\alpha^{-1}|^{w-2}K_1[\ord_p\frac{\alpha^2}{\gamma}]K_2[\ord_p\beta,\ord_p\gamma]\\ &\times\left(1 + |p|^{w-2} + \cdots + |p|^{(w-2)\ord_p \frac{\alpha^2}{\gamma}}\right) \\ & \times \left(1 + |p|^{w-2s-2} + \cdots + |p|^{(w-2s-2)\ord_p\beta}\right)\,d\alpha d\beta d\gamma.\end{align*}
The above integral is over $\alpha, \beta, \gamma$ in $\mathcal{O}_F \cap F^{\times}$ with $|\gamma| \leq |\alpha|$.  Similarly, we obtain
\begin{align*} \zeta(w-2s) \zeta(w-1) I_{|\alpha| < |\gamma|}(W_1,W_2,s,w) =& \int_{\alpha,\beta,\gamma} |\beta^2 \gamma|^{s}|\alpha\gamma^{-1}|^{w-2}K_1[\ord_p\frac{\alpha^2}{\gamma}]K_2[\ord_p \beta,\ord_p \gamma] \\ & \times \left(1 + |p|^{w-2} + \cdots + |p|^{(w-2)\ord_p\gamma}\right) \\ & \times\left(1 + |p|^{w-2s-2} + \cdots + |p|^{(w-2s-2)\ord_p\frac{\beta \gamma}{\alpha}}\right)\,d\alpha d\beta d\gamma.\end{align*}
This integral is over $\alpha, \beta, \gamma$ in $\mathcal{O}_F \cap F^{\times}$ with $|\beta| \leq |\alpha/\gamma| < 1$.

Set $a = \ord_p\alpha$, $b = \ord_p\beta$, and $c = \ord_p\gamma$, and define $U = |p|^{w-2}$, $V = |p|^{s}$.  Finally, set $I'(W_1,W_2,s,w) = \zeta(w-1)\zeta(w-2s)I(W_1,W_2,s,w)$.  Then we have computed
\[I'(W_1,W_2,s,w) = I'_{a \leq c}(W_1,W_2,s,w) + I'_{c < a}(W_1,W_2,s,w),\]
where
\begin{align*} I_{a \leq c}'(W_1,W_2,s,w) &= \sum_{a,b,c \geq 0, a \leq c \leq 2a} U^{c-a}V^{2b+c}A_1[2a-c]B_2[b,c] (1+U + \cdots + U^{2a-c})\\ & \qquad \qquad \times \left(1 + \frac{U}{V^2} + \cdots + \left(\frac{U}{V^2}\right)^{b}\right)\end{align*}
and
\begin{align*} I_{c < a}'(W_1,W_2,s,w) &= \sum_{a,b,c \geq 0, b \geq a -c > 0} U^{a-c}V^{2b+c}A_1[2a-c]B_2[b,c] (1+U + \cdots + U^{c})\\ & \qquad \qquad \times \left(1 + \frac{U}{V^2} + \cdots + \left(\frac{U}{V^2}\right)^{b+c-a}\right).\end{align*}
This completes the proof of Theorem \ref{thm:localintegral}.
\section{Unramified calculation} \label{sec:unramified}

We first work with the right-hand side of the expression in (\ref{eqn:mainthm}).  By expanding the $L$-functions using symmetric algebra decompositions in Section \ref{subsec:symmalg} and evaluating the product of the $L$-functions using a Pieri rule in Section \ref{subsec:pieri}, we arrive at an explicit expression (\ref{eqn:lprod}) for the product of $L$-functions.

Our comparison of the right-hand side of (\ref{eqn:lprod}) with (\ref{eqn:localintegral}) is in three stages.  We first write the right-hand side of (\ref{eqn:localintegral}) in the form
\begin{align}\label{eqn:multdef} &\sum_{\substack{a,b,c \ge 0\\a \le c \le 2a}} \sum_{x,y \ge 0} m(x,y,a,b,c)U^{c-a+x}V^{c+2y}A_1[2a-c]B_2[b,c] \\+& \sum_{\substack{a,b,c \ge 0\\c < a \le b+c}} \sum_{x,y \ge 0} m(x,y,a,b,c)U^{a-c+x}V^{2a-c+2y}A_1[2a-c]B_2[b,c] \nonumber\end{align}
for an explicit coefficient function $m(x,y,a,b,c)$ in Section \ref{subsec:locintside}.  We then do the same to obtain a coefficient function $n(x,y,a,b,c)$ for the the right-hand side of (\ref{eqn:lprod}) in Section \ref{subsec:lfuncside}, and finally show that $m=n$ in Section \ref{subsec:compare}.

\subsection{Symmetric algebra decompositions} \label{subsec:symmalg}

In this section, we collect results from \cite{brion} and \cite[Appendix]{gpsr} that prove the following proposition, which writes each of the $L$-functions on the right-hand side of (\ref{eqn:mainthm}) as a series expansion in characters of the dual group $\SL_2(\mathbf{C})\times \Spin_5(\mathbf{C})$.
\begin{prop}
We have
\begin{align} \label{eqn:symmalg} &\frac{L(\pi_1 \times \pi_2, \mathbf{1} \times \mathrm{Std},w-2)L(\pi_1 \times \pi_2, \mathrm{Std} \times \mathrm{Spin},s)}{\zeta(2w-4)\zeta(2s)}\\ = &\(\sum_{k \ge 0} U^k B_2[k,0]\)\(\sum_{m,n\ge 0} V^{m+2n}A_1[m]B_2[n,m]\).\nonumber \end{align}
\end{prop}

\begin{proof}
Equivalently, we must show that
\begin{equation}\label{eqn:symalg1}\sum_{\ell \ge 0} \mathrm{Sym}^\ell(B_2[1,0])U^\ell = \zeta(U^2)\sum_{k \ge 0} B_2[k,0]U^k\end{equation}
and
\begin{equation}\label{eqn:symalg2}\sum_{\ell \ge 0} \mathrm{Sym}^\ell(A_1[1]B_2[0,1])V^\ell = \zeta(V^2)\sum_{m,n \ge 0} A_1[m]B_2[n,m]V^{n+2m}.\end{equation}
The symmetric algebra decomposition for $B_2[1,0]$ is given by the table in \cite{brion}.  It has a generator of weight 2 given by the trivial representation, so we may factor it out to get the zeta factor $\zeta(U^2)$, giving (\ref{eqn:symalg1}).

For the second equality, (A.1.3) of \cite[Appendix]{gelbartPSR} calculates
\[\mathrm{Sym}^\ell(A_1[1]B_2[0,1]) = \sum_{\substack{0 \le j \le \ell\\j \equiv \ell\ (\textrm{mod }2)}} \sum_{0 \le i \le \frac{j}{2}} A_1[j-2i]B_2[i,j-2i].\]
We have $\zeta(V^2)^{-1} = 1-V^2$, so to prove (\ref{eqn:symalg2}), we need to check the identity of power series expansions
\[(1-V^2)\sum_{\ell \ge 0} \sum_{\substack{0 \le j \le \ell\\j \equiv \ell\ (\textrm{mod }2)}} \sum_{0 \le i \le \frac{j}{2}} A_1[j-2i]B_2[i,j-2i]V^\ell = \sum_{m,n \ge 0} A_1[m]B_2[n,m]V^{n+2m}.\]
To verify this equality, it suffices to show that the coefficient of $A_1[m]B_2[n,m]$ on the left is $V^{n+2m}$.  We distribute the multiplication on the left so that the summand becomes $A_1[j-2i]B_2[i,j-2i](V^\ell-V^{\ell+2})$.  If $A_1[j-2i]B_2[i,j-2i] = A_1[m]B_2[n,m]$, then $j-2i = m$ and $i = n$, so $j = m+2n$.  The possible values of $\ell$ are all values greater than or equal to $j$ and congruent to $j$ modulo 2, so the coefficient is $\sum_{k \ge 0} (V^{j+2k} - V^{j+2k+2})$, which telescopes to $V^j = V^{n+2m}$.
\end{proof}

\subsection{Orthogonal Pieri rules} \label{subsec:pieri}

In this section, we apply orthogonal Pieri rules in order to expand the right hand side of (\ref{eqn:symmalg}).  The results may be summarized in the following proposition.

\begin{prop} \label{prop:pieri}

We have
\begin{align} &\(\sum_{k \ge 0} U^k B_2[k,0]\)\(\sum_{m',n\ge 0} V^{m'+2n}A_1[m']B_2[n,m']\)\nonumber\\
  \label{eqn:lprod} =&\sum_{k,m,n \ge 0} \sum_{\epsilon \in \set{0,1}} \sum_{\alpha,\beta,i} U^k V^{2m+2n} A_1[2m] B_2[2\alpha+k-n-2m-\epsilon-2i,2\beta+2i]\\
	+&\sum_{k,m,n \ge 0} \sum_{\epsilon \in \set{0,1}} \sum_{\alpha,\beta,i} U^k V^{2m+2n+1} A_1[2m+1]B_2[2\alpha+k-n-2m-\epsilon-2i,2\beta+2i+1], \nonumber
\end{align}
where the sums over $\alpha,\beta,$ and $i$ on the right-hand side are subject to the conditions
\begin{align}
  \label{eqn:cond1} m \le & \alpha \le m+n\\
  \label{eqn:cond2} \underline{\epsilon} \le & \beta \le m\\
	\label{eqn:cond3} 0 \le & \, i \le \min(\alpha-\beta,k-2m-n+\alpha+\beta-\epsilon).
\end{align}
For brevity, we have written $\underline{\epsilon}$ to denote $\epsilon$ for the first sum on the right-hand side of (\ref{eqn:lprod}) and 0 for the second sum.

\end{prop}

To prove the proposition, we apply a Pieri rule found in \cite[\S 3]{bfg2}.  We state the rule only for the case of $B_2$, though the rule applies to any orthogonal group.  Following the notation there, a partition $\lambda = \set{k_1+k_2,k_2}$ is associated to the representation $B_2[k_1,2k_2]$ of $\Spin_5(\mathbf{C})$, which we denote by $\pi(\lambda)$.  We also write $\pi(\Delta,\lambda)$ for the representation $B_2[k_1,2k_2+1]$.
\begin{thm}[{\cite[Propositions 3.3 and 3.4]{bfg2}}] \label{thm:bfg}
Suppose that $\sigma$ and $\pi$ are partitions of length at most 2.  We have the following two facts.
\begin{enumerate}
	\item The multiplicity of $\pi(\sigma)$ in $\pi(\lambda) \otimes B_2[k,0]$ is the number of partitions $\nu$ contained in both $\sigma$ and $\lambda$ with $\lambda\setminus \nu$ and $\sigma \setminus \nu$ horizontal strips such that either $|\lambda \setminus \nu| + |\sigma \setminus \nu| =k$ or $\lambda$ has length 2 and $|\lambda \setminus \nu| + |\sigma \setminus \nu| =k-1$.
	\item The multiplicity of $\pi(\Delta,\sigma)$ in $\pi(\Delta,\lambda) \otimes B_2[k,0]$ is the number of partitions $\nu$ contained in both $\sigma$ and $\lambda$ with $\lambda\setminus \nu$ and $\sigma \setminus \nu$ horizontal strips such that $|\lambda \setminus \nu| + |\sigma \setminus \nu| =k$ or $k-1$.
\end{enumerate}
\end{thm}

\begin{proof}[Proof of Proposition \ref{prop:pieri}]
Let $\lambda = \set{m+n,m}$, so that $\pi(\lambda) = [n,2m]$ and $\pi(\Delta,\lambda)=[n,2m+1]$.  The sum on the left-hand side of (\ref{eqn:lprod}) will contribute to the first or second sum on the right based on whether $m' = 2m$ is even or $m' = 2m+1$ is odd, respectively.  We apply Theorem \ref{thm:bfg} to calculate $\pi(\lambda) \otimes B_2[k,0]$ and $\pi(\Delta,\lambda) \otimes B_2[k,0]$.  Note that the Pieri rules differ only slightly, so we will be able to handle both cases simultaneously.

Write $\nu = \set{\alpha,\beta}$.  Then the condition that $\nu$ is contained in $\lambda$ and $\lambda \setminus \nu$ is a horizontal strip implies $m \le \alpha \le m+n$ and $0 \le \beta \le m$.  In considering possibilities for $\sigma$, let $\epsilon \in \set{0,1}$ be such that $|\lambda \setminus \nu| + |\sigma \setminus \nu| =k-\epsilon$.  Note that $|\lambda \setminus \nu| = 2m+n-\alpha-\beta$.  Then letting $i \ge 0$ be the number of additional blocks in the second row of $\sigma$ compared to $\nu$, we must have $\sigma = \set{\alpha+k-2m-n+\alpha+\beta-\epsilon-i,\beta+i}$.  The condition that $\sigma$ contains $\nu$ implies that $k-2m-n+\alpha+\beta-\epsilon-i \ge 0$.  For $\sigma \setminus \nu$ to be a horizonal strip, we need only require that $i\le \alpha - \beta$.  In the spinor case, this completely determines the $\sigma$ that appear, and we obtain the second sum in Proposition \ref{prop:pieri}.  In the non-spinor case, we note that the length 2 condition translates to $\beta \ge 1$ when $\epsilon = 1$, or in other words, $\beta \ge \epsilon$.  The equality in (\ref{eqn:lprod}) follows.
\end{proof}

\subsection{Character coefficients: local integral side} \label{subsec:locintside}

For each of the sums appearing in (\ref{eqn:multdef}), we would like to write an expression for $m(x,y,a,b,c)$ so that (\ref{eqn:multdef}) is equal to the right-hand side of (\ref{eqn:localintegral}).  Expanding out the product in (\ref{eqn:localintegral}), the coefficient of $A_1[2a-c]B_2[b,c]$ is $U^{c-a}V^c$ multiplied by
\begin{equation} \label{eqn:loccoeffca} \sum_{\substack{0 \le d \le 2a-c \\ 0 \le e \le b \\ 0 \le f}} U^{b+d-e+f}(V^2)^{e+f}\end{equation}
if $c \ge a$ or $U^{a-c}V^{2a-c}$ multiplied by
\begin{equation} \label{eqn:loccoeffac} \sum_{\substack{0 \le d \le c \\ 0 \le e \le -a+b+c \\ 0 \le f}} U^{-a+b+c+d-e+f}(V^2)^{e+f}\end{equation}
if $a < c$.  We next write the summations in (\ref{eqn:loccoeffca}) and (\ref{eqn:loccoeffac}) in the form $\sum_{x,y \ge 0} m(x,y,a,b,c)U^xV^{2y}$.  In what follows, we regard $a,b,$ and $c$ as fixed and suppress them in the notation.

{\bf Case $c \ge a$}: By examining the terms in the sum in (\ref{eqn:loccoeffca}) and using the inequalities on $d,e,$ and $f$, it is clear that if the monomial $U^xV^{2y}$ appears at all, we must have
\begin{equation} \label{eqn:caineq}-2a-b+c \le -x+y \le b, y \ge 0,\textrm{ and }x+y \ge b.\end{equation}
Moreover, we must have
\begin{equation} \label{eqn:cacong} x+y \equiv b \pmod{2} \textrm{ if } 2a=c,\end{equation}
since $x+y = b+2f$ in this case.  It is easy to see that (\ref{eqn:caineq}) and (\ref{eqn:cacong}) are exactly the conditions for $m(x,y)\neq 0$ when $c\ge a$.  We define
\begin{equation}\label{eqn:cadefdelta} \delta = \begin{cases} 0 & \textrm{if }x+y \equiv b \pmod{2}\\ 1 &\textrm{otherwise}.\end{cases}\end{equation}

\begin{lemma}
We have
\begin{equation} \label{eqn:camform} m(x,y) = \begin{cases} 0 & \textrm{ if }(\ref{eqn:caineq})\textrm{ or }(\ref{eqn:cacong}) \textrm{ does not hold} \\ \min\(\floor{\frac{2a-c-\delta}{2}},\frac{b+x-y-\delta}{2},\floor{\frac{2a+b-c-x+y}{2}},b\)+1 & \textrm{ if }y \ge b\textrm{ and }(\ref{eqn:caineq})\textrm{ and }(\ref{eqn:cacong})\textrm{ hold}  \\ \min\(\floor{\frac{2a-c-\delta}{2}},\frac{-b+x+y-\delta}{2},\floor{\frac{2a+b-c-x+y}{2}},y\)+1 & \textrm{ if }y < b\textrm{ and }(\ref{eqn:caineq})\textrm{ and }(\ref{eqn:cacong})\textrm{ hold}. \end{cases}\end{equation}

\end{lemma}
\begin{proof}
We need only check this when (\ref{eqn:caineq}) and (\ref{eqn:cacong}) hold.  The number $m(x,y)$ is the number of solutions $(d,e,f)$ to $U^xV^{2y}=U^{b+d-e+f}(V^2)^{e+f}$ meeting the constraints in (\ref{eqn:loccoeffca}).  Starting with the solution for which $d$ is minimal, the linear conditions $x-b=d-e+f$ and $y=e+f$ force the remaining solutions to be of the form $(d,e,f) + j\cdot (2,1,-1)$ for $j \ge 0$.

To determine the number of solutions, we break into cases.  If $-x+y \ge -b$, then we obtain a solution $(d,e,f) = (\delta,\frac{-x+y+b+\delta}{2},\frac{x+y-b-\delta}{2})$, where $\delta = 0$ or 1 to match the parity of $x+y-b$, which is automatically minimal in $d$.  Note that the inequalities in (\ref{eqn:caineq}) together with $-x+y \ge -b$ guarantee that $e$ and $f$ are in the allowed range, while (\ref{eqn:cacong}) ensures that $d=\delta$ is as well.  The largest value of $j$ for which $(d,e,f) + j\cdot (2,1,-1)$ is simply the largest one for which all three of the variables are within the allowed bounds.  For $d$ this is $\floor{\frac{2a-c-\delta}{2}}$, for $e$ this is $b-\frac{-x+y+b+\delta}{2} = \frac{x-y+b-\delta}{2}$, and for $f$ this is $\frac{x+y-b-\delta}{2}$.  So
\begin{equation} \label{eqn:casexyb} m(x,y) = \min\(\floor{\frac{2a-c-\delta}{2}},\frac{x-y+b-\delta}{2},\frac{x+y-b-\delta}{2}\)+1\end{equation}
in this case.  We note that in the expressions on the right-hand side of (\ref{eqn:camform}), only the first and second terms in the minima are relevant when $-x+y \ge -b$.  Moreover, if $y \ge b$, $\frac{x+y-b-\delta}{2}$ is irrelevant in (\ref{eqn:casexyb}) while if $y < b$, $\frac{x-y+b-\delta}{2}$ is irrelevant.  This proves (\ref{eqn:camform}) in this case.

We now assume that $-x+y < -b$.  In this case, only the third and fourth terms in the minima on the right-hand side of (\ref{eqn:camform}) are relevant.  The solution $(d,e,f) = (x-y-b,0,y)$ is in range and clearly minimal in $d$.  Analyzing the upper bounds of each component as before give
\begin{equation} \label{eqn:casebyx} m(x,y) = \min\(\floor{\frac{2a+b-c-x+y}{2}},b,y\)+1.\end{equation}
As before, an inequality between $y$ and $b$ allows us to exclude one of the three values, giving the equality (\ref{eqn:camform}) again in this case.
\end{proof}

{\bf Case $a > c$}:  The analysis of this case is identical to the previous one, so we simply state the final result.  The exact conditions for $m(x,y) \ne 0$ are that 
\begin{equation} \label{eqn:acineq} a-b-2c \le -x+y \le -a+b+c, y \ge 0,x+y \ge -a+b+c,\textrm{ and, if }c=0, x+y \equiv b+a \pmod{2}.\end{equation}
Define
\begin{equation}\label{eqn:acdefdelta} \delta = \begin{cases} 0 & \textrm{if }x+y \equiv -a+b+c \pmod{2}\\ 1 &\textrm{otherwise}.\end{cases}\end{equation}
We have
\begin{align} \label{eqn:acmform} &m(x,y) = \\\nonumber &\begin{cases} 0 & \textrm{ if }(\ref{eqn:acineq}) \textrm{ does not hold} \\ \min\hspace{-1.5pt}\(\floor{\frac{c-\delta}{2}},\frac{-a+b+c+x-y-\delta}{2},\floor{\frac{-a+b+2c-x+y}{2}},-a+b+c\)+1 & \textrm{ if }y \ge -a+b+c\textrm{ and }(\ref{eqn:acineq})\textrm{ holds}  \\ \min\hspace{-1.5pt}\(\floor{\frac{c-\delta}{2}},\frac{a-b-c+x+y-\delta}{2},\floor{\frac{-a+b+2c-x+y}{2}},y\)+1 & \textrm{ if }y < -a+b+c\textrm{ and }(\ref{eqn:acineq})\textrm{ holds}. \end{cases}
\end{align}

\subsection{Character coefficients: $L$-function side} \label{subsec:lfuncside}

We now work with the right-hand side of (\ref{eqn:lprod}) in order to write it in the form of (\ref{eqn:multdef}).  We write $n(x,y,a,b,c)$ in place of $m(x,y,a,b,c)$ to distinguish it from the function in Section \ref{subsec:locintside}, though we again abbreviate the notation to $n(x,y)$, keeping $a,b,$ and $c$ fixed.

We will simultaneously handle the case of $c$ odd or even by using the following notational conventions.  We use $\underline{\cdot}$ for a term present only if $c$ is even and $\overline{\cdot}$ for a term present only if $c$ is odd.  The former is already utilized in (\ref{eqn:cond2}) above.  We also define $\gamma \in \mathbf{Z}$ by the formula $c = 2\gamma + \overline{1}$.

For fixed $a,b,c$, the coefficient $n(x,y)$ is the number of 7-tuples $(k,m,n,\epsilon,\alpha,\beta,i)$ such that
\begin{itemize}
	\item the inequalities in (\ref{eqn:cond1}), (\ref{eqn:cond2}), and (\ref{eqn:cond3}) are satisfied and
	\item we have $U^kV^{2m+2n+\overline{1}}A_1[2m+\overline{1}]B_2[2\alpha+k-n-2m-\epsilon-2i,2\beta+2i+\overline{1}]=U^{c-a+x}V^{c+2y}A_1[2a-c]B_2[b,c]$ or $U^{a-c+x}V^{2a-c+2y}A_1[2a-c]B_2[b,c]$ depending on whether $c \ge a$ or $a>c$, respectively.
\end{itemize}
Our goal is to calculate $n(x,y)$ more explicitly so that we may compare the result to (\ref{eqn:camform}) or (\ref{eqn:acmform}).  Although we will break into cases based on whether $c \ge a$ or $a > c$, we will do our calculations in such a way that the answer in the latter case can be deduced quickly from the former.

{\bf Case $c \ge a$}:  The equality $U^{c-a+x}V^{c+2y}A_1[2a-c]B_2[b,c]=U^kV^{2m+2n+\overline{1}}A_1[2m+\overline{1}]B_2[2\alpha+k-n-2m-\epsilon-2i,2\beta+2i+\overline{1}]$ allows us to pare down the number of free parameters as follows.  We have
\begin{equation} \label{eqn:detk} k =c-a+x, 2m+\overline{1} = 2a-c, \textrm{ and } 2m+2n+\overline{1} = c+2y,\end{equation}
which, solving for $m$ and $n$ in the second and third equations, (using $c = 2\gamma+\overline{1}$) gives
\begin{equation} \label{eqn:detmn} m = a-\gamma-\overline{1} \textrm{ and } n = \gamma+y-m = -a+c+y.\end{equation}
The relation $2\beta+2i+\overline{1} = c = 2\gamma+\overline{1}$ gives $\beta = \gamma-i$.  So we are reduced to looking for triples $(\epsilon,\alpha,i)$.  We can rewrite the relation
\[2\alpha+k-n-2m-\epsilon-2i = b\]
as
\[2\alpha+k-n-2m-b = 2i+\epsilon.\]
Using the formulas for $k,n$, and $m$ in (\ref{eqn:detk}) and (\ref{eqn:detmn}), this uniquely determines $i$ and $\epsilon$ in terms of $\alpha$ and the givens $a,b,c,x$, and $y$.  In particular, we are reduced to counting the values of a single parameter $\alpha$ that satisfy the three inequalities (\ref{eqn:cond1}), (\ref{eqn:cond2}), and (\ref{eqn:cond3}) (all of which can be written in terms of $a,b,c,x$, and $y$).  We note that $\epsilon$ is equal to 0 or 1 based on the parity of
\[2\alpha+k-n-2m-b \equiv x+y+b \pmod{2},\]
and is thus equal to the $\delta$ of (\ref{eqn:cadefdelta}).  We also note that
\[i = \alpha+\frac{k-n-2m-b-\epsilon}{2} = \alpha+\frac{x-y-2a-b+c+\overline{1}-\epsilon}{2}.\]
Using this, we may now list out the conditions on $\alpha$.  From (\ref{eqn:cond1}) we have
\begin{equation} \label{eqn:alpha1} a-\gamma-\overline{1} \le \alpha \textrm{ and } \alpha \le \gamma + y.\end{equation}
Substituting for $\beta$ in (\ref{eqn:cond2}), we obtain
\[\underline{\epsilon} \le \gamma-\alpha-\frac{x-y-2a-b+c+\overline{1}-\epsilon}{2} \le m.\]
The lower bound rearranges to give the upper bound
\begin{equation} \label{eqn:alpha2} \alpha \le -\underline{\epsilon}+\frac{-x+y+2a+b-\overline{2}+\epsilon}{2} \end{equation}
on $\alpha$.  The upper bound rearranges to give the lower bound
\begin{equation} \label{eqn:alpha3} \frac{-x+y+b+c-\overline{1}+\epsilon}{2} \le \alpha\end{equation}
on $\alpha$.  For (\ref{eqn:cond3}), we first note that the first upper bound gives $i \le \alpha-\gamma+i$ after substituting for $\beta$, which rearranges to
\begin{equation} \label{eqn:alpha4} \gamma \le \alpha. \end{equation}
Substituting for $k,m,n,i$ and $\beta$ in the remaining terms in (\ref{eqn:cond3}) gives
\begin{align*} 0 &\le \alpha+\frac{x-y-2a-b+c+\overline{1}-\epsilon}{2}\\
  &\le (c-a+x)-2(a-\gamma-\overline{1})-(-a+c+y)+\alpha+\gamma - \(\alpha+\frac{x-y-2a-b+c+\overline{1}-\epsilon}{2}\) -\epsilon\\
	&=x-y-2a+c+\overline{1}+\gamma - \(\frac{x-y-2a-b+c+\overline{1}-\epsilon}{2}\) -\epsilon
\end{align*}
Rearranging this gives the lower and upper bounds
\begin{equation} \label{eqn:alpha5} \frac{-x+y+2a+b-c-\overline{1}+\epsilon}{2} \le \alpha \textrm{ and } \alpha \le b+\gamma. \end{equation}
We may now conclude that $n(x,y)$ is the number of integers in the interval of solutions for $\alpha$ permitted by (\ref{eqn:alpha1}), (\ref{eqn:alpha2}), (\ref{eqn:alpha3}), (\ref{eqn:alpha4}), and (\ref{eqn:alpha5}).  We can slightly simplify this situation using the hypothesis $c \ge a$, as the first inequality in (\ref{eqn:alpha5}) is implied by (\ref{eqn:alpha3}) and $c \ge a$, so we only need the latter one.  Moreover, the first inequality in (\ref{eqn:alpha1}) is implied by (\ref{eqn:alpha4}) and $c \ge a$ as well.

{\bf Case $a > c$}:  The inequality $c \ge a$ was only used in the last two sentences of the preceding calculation in order to render the first inequalities in (\ref{eqn:alpha1}) and (\ref{eqn:alpha5}) redundant.  Thus the remainder of the calculations apply equally well in the case $a > c$ once a substitution is made to make the definitions of $x$ and $y$ compatible.  We now have $U^kV^{2m+2n+\overline{1}}A_1[2m+\overline{1}]B_2[2\alpha+k-n-2m-\epsilon-2i,2\beta+2i+\overline{1}]=U^{a-c+x}V^{2a-c+2y}A_1[2a-c]B_2[b,c]$, so $k = a-c+x$ and $2m+2n+\overline{1} = 2a-c+2y$.  Comparing these to (\ref{eqn:detk}), we see that if we replace $x$ with $x+2a-2c$ and $y$ with $y+a-c$ in (\ref{eqn:alpha1}), (\ref{eqn:alpha2}), (\ref{eqn:alpha3}), (\ref{eqn:alpha4}), and (\ref{eqn:alpha5}), we obtain the conditions for the $a>c$ case.  We obtain the following conditions.
\begin{gather}
  \label{eqn:acalpha1} a-\gamma-\overline{1} \le \alpha,\alpha \le \gamma + y+a-c, \alpha \le b+\gamma \textrm{ and } \gamma \le \alpha\\
	\label{eqn:acalpha2} \alpha \le -\underline{\epsilon}+\frac{-x+y+a+b+c-\overline{2}+\epsilon}{2}\\
	\label{eqn:acalpha3} \frac{-x+y-a+b+2c-\overline{1}+\epsilon}{2} \le \alpha \textrm{ and } \frac{-x+y+a+b-\overline{1}+\epsilon}{2} \le \alpha
\end{gather}
Since $a > c$, the second inequality in (\ref{eqn:acalpha3}) implies the first inequality in (\ref{eqn:acalpha3}).  Moreover, $\gamma \le \alpha$ is implied by $a-\gamma-\overline{1} \le \alpha$.

\subsection{Comparison of coefficient functions} \label{subsec:compare}

We now have an explicit description of $m(x,y,a,b,c)$ and sets of inequalities that determine $n(x,y,a,b,c)$.  In this section, we prove that $m(x,y,a,b,c)=n(x,y,a,b,c)$ in all cases.  We will continue to drop $a,b,$ and $c$ from the notation.

\begin{prop} \label{prop:ineq}
We have $m(x,y) = n(x,y)$.
\end{prop}

\begin{proof}

We break into cases based on whether $c \ge a$ or $a > c$.

{\bf Case $c \ge a$}:  Recall that the conditions for non-vanishing of $m$ are given by (\ref{eqn:caineq}) and (\ref{eqn:cacong}).  So if $m$ vanishes, we must have
\begin{enumerate}
	\item $-x+y > b$,
	\item $-x+y < -2a-b+c$,
	\item $y < 0$,
	\item $x+y < b$, or
	\item $2a=c$ and $x+y \not\equiv b \pmod{2}$.
\end{enumerate}
In each of these cases, we verify that $n$ vanishes, while simultaneously proving that $m=n$ under a related condition.  After checking one additional case, the combination of the resulting equalities will prove the proposition in all cases.
\begin{enumerate}
	\item The inequality (\ref{eqn:alpha3}) combined with the second inequality in (\ref{eqn:alpha5}) gives $\gamma + \frac{-x+y+b+\epsilon}{2} \le \alpha \le b+\gamma$.  It follows that if $-x+y > b$, $n(x,y)=0$, and $n(x,y) \le \frac{x-y+b-\epsilon}{2}+1$ otherwise.

We now claim that if $m(x,y) = \frac{x-y+b-\epsilon}{2}+1 \ne 0$, then $m(x,y) = n(x,y)$.  Note that examination of (\ref{eqn:camform}) implies that $y \ge b$ in this case.  Since we have already checked that $n \le m$ under this condition, it suffices to check that every upper bound for $\alpha$ is at least $b+\gamma$ and every lower bound is at most $\gamma + \frac{-x+y+b+\epsilon}{2}$.

We first handle the lower bound $\gamma$ from (\ref{eqn:alpha4}).  To see that $\gamma \le \gamma + \frac{-x+y+b+\epsilon}{2}$ or
\begin{equation} \label{eqn:lowerineq} 0 \le \frac{-x+y+b+\epsilon}{2}, \end{equation}
observe that $m(x,y) = \frac{x-y+b-\epsilon}{2}+1$ implies $\frac{x-y+b-\epsilon}{2}\le b$ or $0 \le \frac{-x+y+b+\epsilon}{2}$.  Since $y \ge b$ the claim for the upper bound $\gamma + y$ is trivial.  Finally, we need $b + \gamma \le -\underline{\epsilon}+\frac{-x+y+2a+b-\overline{2}+\epsilon}{2}$.  We can rearrange this inequality to
\begin{equation} \label{eqn:case1ineq} 0 \le \frac{-x+y+2a-b-c+\epsilon-\overline{1}-2\underline{\epsilon}}{2}. \end{equation}
We will deduce this from the fact that since $m(x,y) = \frac{x-y+b-\epsilon}{2}+1$, we have $\frac{x-y+b-\epsilon}{2} \le \floor{\frac{2a-c-\epsilon}{2}}$.  Since the left-hand side of this inequality is an integer, we have
\[0 \le \floor{\frac{-x+y+2a-b-c}{2}},\]
and moreover the parity of this numerator is that of $c-\epsilon$.  Comparing with (\ref{eqn:case1ineq}), it suffices to verify that if $c-\epsilon$ is odd, then $\epsilon-\overline{1}-2\underline{\epsilon} \ge -1$, and that if $c-\epsilon$ is even, then $\epsilon-\overline{1}-2\underline{\epsilon} \ge 0$.  It is now easy to check that in each of four cases for the parity of $\epsilon$ and of $c$, the relevant inequality is satisfied.  In fact, we have equality:
\begin{equation} \label{eqn:epsfact} \epsilon-\overline{1}-2\underline{\epsilon} = -1\textrm{ or }\epsilon-\overline{1}-2\underline{\epsilon} = 0\end{equation}
depending on the parity of $c-\epsilon$.  This will be useful later.
	\item The inequalities (\ref{eqn:alpha2}) and (\ref{eqn:alpha4}) give $\gamma \le \alpha \le -\underline{\epsilon}+\frac{-x+y+2a+b-\overline{2}+\epsilon}{2}$. The difference between the outer expressions is $-\underline{\epsilon}+\frac{-x+y+2a+b-c-\overline{1}+\epsilon}{2}$.  Considering the cases for $c \pmod{2}$ and $\epsilon$, it follows that if this difference is nonnegative, $n(x,y) \le \floor{\frac{2a+b-c-x+y}{2}}+1$.  Now suppose that $-x+y < -2a-b+c$, so that $-\underline{\epsilon}+\frac{-x+y+2a+b-c-\overline{1}+\epsilon}{2} \le -\underline{\epsilon}+\frac{-1-\overline{1}+\epsilon}{2}$.  If $\epsilon = 0$, the right-hand side is negative, so $n(x,y)=0$.  If $\epsilon = 1$, either $-\underline{\epsilon}$ or $-\frac{\overline{1}}{2}$ is strictly negative, so we again find $n(x,y)=0$.

We now claim that if $m(x,y) = \floor{\frac{2a+b-c-x+y}{2}}+1 \ne 0$, then $m(x,y)=n(x,y)$.  Similarly to the previous case, it suffices to check that every lower bound on $\alpha$ is at most $\gamma$ and every upper bound is at least $-\underline{\epsilon}+\frac{-x+y+2a+b-\overline{2}+\epsilon}{2}$.  For the lower bound, we need $\frac{-x+y+b+c-\overline{1}+\epsilon}{2} \le \gamma$ or, rearranging,
\begin{equation} \label{eqn:lowerineqflip} 0 \le \frac{x-y-b-\epsilon}{2}.\end{equation}
The condition on $m$ gives us $\floor{\frac{2a+b-c-x+y}{2}} \le \floor{\frac{2a-c-\epsilon}{2}}$, which implies the result since the two sides have the same parity.

For the upper bound $b+\gamma$, the needed inequality rearranges to (\ref{eqn:case1ineq}) with the inequality flipped, or
\begin{equation} \label{eqn:case1ineqflip} 0 \ge \frac{-x+y+2a-b-c+\epsilon-\overline{1}-2\underline{\epsilon}}{2}. \end{equation}
The condition on $m$ gives $\floor{\frac{2a+b-c-x+y}{2}} \le b$ or $\floor{\frac{2a-b-c-x+y}{2}} \le 0$.  Since the parity of $2a-b-c-x+y$ is that of $c - \epsilon$, combining this with the fact (\ref{eqn:epsfact}) above proves (\ref{eqn:case1ineqflip}).  The argument for the other upper bound $y+\gamma$, for which we need
\begin{equation} \label{eqn:case1ineqflipyb} 0 \ge \frac{-x-y+2a+b-c+\epsilon-\overline{1}-2\underline{\epsilon}}{2}, \end{equation}
is identical; simply replace $\floor{\frac{2a+b-c-x+y}{2}} \le b$ with $\floor{\frac{2a+b-c-x+y}{2}} \le y$.
	\item The combination of the second inequality in (\ref{eqn:alpha1}) with (\ref{eqn:alpha4}) is $\gamma \le \alpha \le y+\gamma$.  It follows immediately that $n(x,y)=0$ if $y <0$ and that $n(x,y) \le y+1$ otherwise.

We now claim that if $m(x,y) = y+1 \ne 0$, then $m(x,y) = n(x,y)$.  As before, we need the other lower bound to be no more than $\gamma$ and the other upper bounds to be at least $y+\gamma$.  We first note that if $m(x,y)=y+1$, then $y \le b$.  The inequality $y + \gamma \le b + \gamma$ follows.  The needed comparison between $\gamma$ and the lower bound in (\ref{eqn:alpha2}) is equivalent to (\ref{eqn:lowerineqflip}).  Since $m(x,y) = y+1$, $y \le \frac{-b+x+y-\epsilon}{2}$, which rearranges to (\ref{eqn:lowerineqflip}).  The comparison with (\ref{eqn:alpha4}) gives (\ref{eqn:case1ineqflipyb}) flipped, or
\begin{equation} \label{eqn:case1ineqyb} 0 \le \frac{-x-y+2a+b-c+\epsilon-\overline{1}-2\underline{\epsilon}}{2}. \end{equation}
This follows using the case analysis from (\ref{eqn:epsfact}) from $y \le \floor{\frac{2a+b-c-x+y}{2}}$.
	\item The combination of the second inequality in (\ref{eqn:alpha1}) with (\ref{eqn:alpha3}) gives $\frac{-x + y+b+c-\overline{1}+\epsilon}{2} \le \alpha \le \gamma+y$.  The difference of the outer terms is $\frac{x+y-b-\epsilon}{2}$, proving that $n(x,y)=0$ if $x+y < b$ and that $n(x,y) \le \frac{-b+x+y-\delta}{2}$ otherwise.

Conversely, if $m(x,y) = \frac{-b+x+y-\epsilon}{2} + 1 \ne 0$, then we claim $m=n$.  In the case we need to check $\gamma + y \le \gamma + b$, (\ref{eqn:lowerineq}), and (\ref{eqn:case1ineqyb}).  The first of these follows from the assumption and examination of (\ref{eqn:camform}), the second follows from $\frac{-b+x+y-\epsilon}{2} \le y$, and the third follows from $\frac{-b+x+y-\epsilon}{2} \le \floor{\frac{2a-c-\epsilon}{2}}$ using the case analysis from (\ref{eqn:epsfact}) as above.
	\item The combination of (\ref{eqn:alpha2}) and (\ref{eqn:alpha3}) gives $\frac{-x+y+b+c-\overline{1}+\epsilon}{2} \le \alpha \le -\underline{\epsilon}+\frac{-x+y+2a+b-\overline{2}+\epsilon}{2}$.  The difference of the outer terms is $\frac{2a-c-\overline{1}-2\underline{\epsilon}}{2}$.  If $2a=c$, then it follows that $n(x,y)=0$ unless $x+y\equiv 2 \pmod{2}$.  Moreover, checking the cases for the parity of $c$ and $\epsilon$, we obtain the inequality $n(x,y) \le \floor{\frac{2a-c-\delta}{2}} + 1$ when the right-hand side is nonnegative.

Conversely, if $m(x,y) = \floor{\frac{2a-c-\epsilon}{2}} + 1 \ne 0$, we claim that $m=n$.  In this case, we need to check (\ref{eqn:lowerineq}), (\ref{eqn:case1ineqflip}), and (\ref{eqn:case1ineqflipyb}).  These follow, respectively, from $\floor{\frac{2a-c-\epsilon}{2}} \le \floor{\frac{2a+b-c-x+y}{2}}$, $\floor{\frac{2a-c-\epsilon}{2}} \le \frac{b+x-y-\epsilon}{2}$, and $\floor{\frac{2a-c-\epsilon}{2}} \le \frac{-b+x+y-\epsilon}{2}$ using arguments as above.
\end{enumerate}
We finally observe that the inequality $n(x,y) \le b+1$ follows immediately from (\ref{eqn:alpha4}) and the second inequality in (\ref{eqn:alpha5}).  If $m(x,y) = b+1$, we claim that $m=n$.  For this, we need to check $\gamma + y \ge \gamma + b$, (\ref{eqn:lowerineqflip}), (\ref{eqn:case1ineq}).  The first follows from $m(x,y) = b+1$ and (\ref{eqn:camform}).  The second and third follow, respectively, from $b \le \frac{b+x-y-\epsilon}{2}$ and $b \le \floor{\frac{2a+b-c-x+y}{2}}$ by arguments as above.

Since we have checked equality for each of the possibilities for $m(x,y)$ in (\ref{eqn:camform}), we deduce that $n(x,y) = m(x,y)$.

{\bf Case $a > c$}: The inequalities on $\alpha$ in this case were obtained from the ones in the $c \ge a$ case by substitution.  Recall that two inequalities were redundant in the $c \ge a$ case, while two others are redundant in the $a > c$ case.  If one replaces everywhere in the argument for the $c\ge a$ case the inequality $\gamma \le \alpha$ with the inequality $a-\gamma-\overline{1} \le \alpha$, replaces the inequality $\frac{-x+y+b+c-\overline{1}+\epsilon}{2} \le \alpha$ by $\frac{-x+y+2a+b-c-\overline{1}+\epsilon}{2}$, and applies the substitution in $x$ and $y$, the result is the proof that $m(x,y) = n(x,y)$ in the $a > c$ case.
\end{proof}

\section{First term identities and poles of multivariate Eisenstein series} \label{sec:poles}
We obtain a first term identity between two non-Siegel Eisenstein series on $\GSp_6$ and apply it to study poles of $L$-functions.  In Section \ref{subsec:archimedean}, we verify a result concerning Archimedean principal series by adapting the proof of a similar result of Jiang \cite{jiang} on degenerate principal series representations of $\mathrm{Sp}_8(\R)$.  By Jiang's theorem on first term identities \cite{jiang2}, which relies only on this Archimedean result, this yields Proposition \ref{FTQ} in Section \ref{subsec:res} below.  We use it to prove Proposition \ref{propB:intro} of the introduction in \ref{subsec:poles}.  The proof requires a calculation to control the integrals at ramified and Archimedean places, which is carried out in Proposition \ref{prop:ramcalc}.  Finally, we explain in Section \ref{subsec:nopole} how to deduce a stronger result than Proposition \ref{propB:intro} by using the integral of Bump, Friedberg, and Ginzburg on $\GSp_4$ \cite{bfg2}.

\subsection{Degenerate Archimedean principal series} \label{subsec:archimedean}

We will need to prove a result on degenerate principal series of a particular parabolic subgroup of $\Sp_6$ in order to apply Jiang's first term identity.  We use the basis and symplectic form from Section \ref{subsec:global}.  Let $R$ be the maximal parabolic stabilizing $\gen{f_1,f_2}$; its Levi $M_R$ is isomorphic to $\GL_2 \times \Sp_2$, and its modulus character $\delta_R(m)$ on $m=(m_1,m_2) \in R_M$ is given by $|\det m_1|^5$.  We have the decomposition $R = M_R U_R$, where $U_R$ is the unipotent radical of $P$.  We also write $W_R$ for the Weyl group of $M_R$, which embeds in $W_G$.  For $p \in R(\mathbf{R})$, write
\begin{equation} \label{eqn:levidec} p = mn = (m_1,m_2)n.\end{equation}
Define the normalized principal series
\[I_R(s) = \set{f:\Sp_6(\mathbf{R}) \rightarrow \mathbf{C}| f(pg) = |\det m_1|^{s + \frac{5}{2}} f(g) \textrm{ for all }p \in R(\mathbf{R}), g \in G(\mathbf{R})}.\]

Then the following is the $n=3$, $r=2$ case of \cite[Assumption/Proposition]{jiang2}, which Jiang assumes as a hypothesis in order to prove a first-term identity.
\begin{prop} \label{prop:assume}
For $s > \frac{1}{2}$, $I_R(s)$ has a unique irreducible quotient and is generated by the spherical section. \end{prop}

Jiang \cite[Chapter 4, Theorem 1.2.1]{jiang} proves a version of Proposition \ref{prop:assume} for the analogous Eisenstein series to $R$ on $\GSp_8$.  The remainder of this section will adapt his arguments to the case at hand.  We emphasize that once we make various relevant calculations for our case, the proof is nearly identical to Jiang's \cite[\S4.2.1, \S4.2.2, \S4.2.3]{jiang}.  We include it for completeness.  Most of the notations below are taken from \emph{loc. cit.}

\subsubsection{Intertwining numbers}

Write $G = \Sp_6(\mathbf{R})$.  Also write $I_R^\mathrm{sm}(s)$ for the space of smooth elements of $I_R(s)$ with compact support modulo $R(\mathbf{R})$.  We also write $K$ simply for the usual maximal compact subgroup of $G$, isomorphic to $U(3)$, and let $I_{R,\infty}^\mathrm{sm}(s)$ denote the subspace of $I_R^\mathrm{sm}(s)$ consisting of $K$-finite functions.  This is the Harish-Chandra model of $I_R(s)$.  We will assume for now that
\begin{equation} \label{eqn:dim1} \dim_\mathbf{C} \hom_G(I_{R,\infty}^\mathrm{sm}(s),I_{R,\infty}^\mathrm{sm}(-s)) =1,\end{equation}
and deduce Proposition \ref{prop:assume} following \cite[\S 4.2.1]{jiang}.  Afterwards, we will prove (\ref{eqn:dim1}).

The proof of Proposition 2.1.2 in Chapter 4 of \cite{jiang}, which assumes the analogue of (\ref{eqn:dim1}), applies equally well to our case once we prove the analogue of Theorem 2.1.2, which shows that the process of taking the contragredient of a principal series representation is equivalent to the adjoint action of a certain element $\delta$.  Following Jiang, we begin by stating Sugiura's classification of conjugacy classes of Cartan subgroups of $\Sp_6$.  We will then study the adjoint action of $\delta$ on these classes and then read off the effect on principal series using the Harish-Chandra classification.

Let $A_1=\set{\diag(t,1,1,1,1,t^{-1})|t \in \mathbf{R}^\times}$.  Define the group $T_{S^1}=\set{\mm{a}{b}{-b}{a}| a^2+b^2=1}$.  We define two conjugacy classes of embeddings of $T_{S^1}$ inside $\Sp_6$ as follows.
\begin{enumerate}
	\item Up to conjugacy, we write $T_1$ for the embedded copy of $T_{S^1}$ defined by
	\[T_1 = \set{\diag(1,t,t^{-1},1):t \in T_{S^1}},\]
	where this matrix has blocks of sizes $1,2,2$, and $1$ down the diagonal.
	\item $T_2$ corresponds to the maximal compact of $\Sp_2$ acting on a symplectic subspace of $V_6$ of dimension 2.
\end{enumerate}
For instance, the conjugacy classes associated to $A_1 \times T_1$ and $A_1 \times A_1 \times T_2$ are of matrices of the form
\[\(\begin{array}{cccc} t & & & \\ & m_1 & & \\ & & {}^tm_1 & \\ & & & t^{-1}\end{array}\) \textrm{ and } \(\begin{array}{ccc}\diag(t,t') & & \\ & m_2 & \\ &  & \diag(t^{\prime-1},t^{-1})\end{array}\),\]
respectively, where $m_1 \in T_1$, $m_2 \in T_2$, and the second matrix is written in $2 \times 2$ block form.  The full list of representatives is
\[A_1^3, A_1^2\times T_2, A_1 \times T_1, A_1 \times T_2^2, T_2 \times T_1,\textrm{ and }T_2^3.\]
Then the Cartan involution $\theta$ given by $g \mapsto {}^tg^{-1}$ fixes each representative, as does the adjoint action of the matrices $\delta^{\pm} = \mm{}{\pm \mathbf{1}_3'}{\mathbf{1}_3'}{} \in \GSp_6$, where $\mathbf{1}_\cdot'$ denotes the antidiagonal 1's matrix.  We note also that the induced action of $\delta^+$ on each $A_1$ or $T_2$ factor is via $h \mapsto h^{-1}$, while $\delta^+$ acts on the $T_1$ factors by $h \mapsto h$.  On the other hand, $\delta' = \mm{}{-\mathbf{1}_2}{\mathbf{1}_2}{} \in K$ acts on the $T_1$ factors by $h \mapsto h^{-1}$.

Vogan and Przebinda prove a relationship between $\pi$ and its attached Harish-Chandra character data and set of lowest $K$ types that is equivariant for the action of automorphisms of $G$ and the process of taking contragredients; see Chapter 4, Theorem 2.1.1 of \cite{jiang} for a concise formulation.  It follows from the above description of Cartan subgroups and action of $\delta^{\pm}$ that each $\pi$ is brought to its contragredient by the adjoint action of $\delta^+$ (or $\delta = \delta^+ \delta^-$, since $\delta^- \in K$).  This is because the $T_1$ factors are already brought to their inverses by a conjugation in $K$ and are fixed by the adjoint action of $\delta^+$, while the $A_1$ and $T_2$ factors are brought to their inverses by $\delta^+$.  In particular, the proof of Proposition 2.1.2 in Chapter 4 of \cite{jiang} now applies in exactly the same way to show that (\ref{eqn:dim1}) implies Proposition \ref{prop:assume}.

\subsubsection{Bruhat decomposition}

We next address the proof of (\ref{eqn:dim1}).  We follow Jiang \cite[\S 4.2.2]{jiang} very closely.  For this we study certain spaces of distributions $T: C_c^\infty(G) \rightarrow \mathbf{C}$, where $C_c^\infty(G)$ denotes smooth functions of compact support on $\Sp_6(\mathbf{R})$.  The space $C_c^\infty(G)$ is given the Schwartz topology as defined in \cite[pg.\ 479]{warner}.  (Also see \cite[\S4]{garrett3} for a more concrete definition.)  We define $C_c^\infty(G)^\vee$ to be the space of continuous linear functionals on $C_c^\infty(G)$.  If $H_1,H_2$ are subgroups of $G$, we write $T^{(h_1,h_2)}$, $(h_1,h_2) \in H_1 \times H_2$ for the function $\phi \mapsto T((h_1,h_2)^{-1}\circ \phi)$, where $((h_1,h_2)^{-1}\circ \phi)(g) = \phi(h_1gh_2^{-1})$.  We say that $T$ has support in a closed subset $S \subseteq G$ if for each $\phi \in C_c^\infty(G \setminus S)$, $T(\phi)=0$.  We define
\[\mathbf{T}_s = \set{T \in C_c^\infty(G)^\vee | T^{(p,q)} = |m_1(pq)|^{s-\frac{5}{2}} T \textrm{ for all }p,q \in R},\]
where $m_1(pq)$ is the matrix $m_1$ appearing in the Levi decomposition (\ref{eqn:levidec}) of the element $pq$.  The dimension $\dim \mathbf{T}_s$ is exactly equal to $\dim_\mathbf{C} \hom_G(I_{R,\infty}^\mathrm{sm}(s),I_{R,\infty}^\mathrm{sm}(-s))$, due to the existence and properties of a trace map taking an intertwiner to a distribution; see \cite[Theorem 5.3.2.1]{warner}.

Recall that we have an embedding $\GSp_2 \boxtimes \GSp_2 \boxtimes \GSp_2 \inj \GSp_6$.  If we write $M_1 \boxtimes M_2 \boxtimes M_3$ for $2 \times 2$ matrices $M_i$, we mean the matrix obtained by having $M_i$ act on $\gen{e_i,f_i}$ in $V_6$.  Consider the Bruhat decomposition $G = \cup_w R w R$.  Write $\mathbf{1}_k'$ to be the antidiagonal 1's matrix as above and  define $\mathbf{J}_2 = \mm{0}{-1}{1}{0}$.  A set of Weyl group representatives are $w_{(i,i)} = \mathbf{1}_2^{\boxtimes i} \boxtimes \mathbf{J}_2^{\boxtimes 2-i} \boxtimes \mathbf{1}_2$ for $i \in \set{1,2}$, $w_{(1,2)} = \diag(1,\mathbf{1}_2',\mathbf{1}_2',1)$,
\[w_{(0,0)} = \(\begin{array}{c|c|c|c|c|c} & & & & 1 &\\ \hline & & & & &1\\ \hline  & & 1 & & &\\ \hline & & & 1 & & \\ \hline -1 & & & & &\\ \hline & -1 & &  & &\end{array}\), \textrm{ and } w_{(0,1)}=\(\begin{array}{c|c|c|c|c|c} & & 1 & & &\\ \hline & & & & 1 &\\ \hline 1 & & & & &\\ \hline & & & & & 1\\ \hline & -1 & & & &\\ \hline & & & 1 & &\end{array}\).\]
We write $O_{(i,j)} = R w_{(i,j)} R$ and set $\Omega_{(i,i)} = \Omega_{(i-1,i)} \cup O_{(i,i)}$ and $\Omega_{i,i+1} = \Omega_{(i,i)}\cup O_{(i,i+1)}$.  Considering $i+j$ gives a linear chain of containments of open subsets $\Omega_{(i,j)}$ of $G$.  Then define
\begin{align*}
  \mathbf{T}_{(i,j)} &= \set{T \in \mathbf{T}_s| \mathrm{Supp}(T) \subseteq \ol{O_{(i,j)}}},\\
	\mathbf{T}_s[O_{(i,j)}] &= \set{T|_{\Omega_{(i,j)}}|T \in \mathbf{T}_{(i,j)}}, \textrm{ and}\\
	\mathbf{T}_s(O_{(i,j)}) &= \set{T\in C_c^\infty(\Omega_{(i,j)})^\vee \right| T^{(p,q)} = |\det m_1(pq)|^{s-\frac{5}{2}}T, \textrm{ for }p,q \in R(\mathbf{R}), \mathrm{Supp}(T) \subseteq O_{(i,j)}\left.}.
\end{align*}
Then $\mathbf{T}_s[O_{(i,j)}] \subset \mathbf{T}_s(O_{(i,j)})$ and
\begin{equation} \label{eqn:orbitineq} \dim \mathbf{T}_s \le \sum_{i,j} \dim \mathbf{T}_s[O_{(i,j)}] \le \sum_{i,j} \dim\mathbf{T}_s(O_{(i,j)}).\end{equation}

\subsubsection{Bruhat estimates}

Let $R_{(i,j)} = R \cap w_{(i,j)}^{-1} Rw_{(i,j)}$ and let $\delta_{(i,j)} = \delta_{R_{(i,j)}}$.  Write $\mathfrak{sp}_{6,\mathbf{C}}$ and $\mathfrak{r}_\mathfrak{C}$ for the complexified Lie algebras of $\Sp_6$ and $R$.  Define $\Xi_{(i,j)}^n$ to be the twist of the $n^\mathrm{th}$ symmetric power of the adjoint representation of $R_{(i,j)}(\mathbf{R})$ on the quotient space
\[\mathfrak{O}_{(i,j)}=\mathfrak{sp}_{6,\mathbf{C}}/(\mathfrak{r}_\mathfrak{C} + \mathrm{ad}(w_{(i,j)}^{-1})\mathfrak{r}_\mathfrak{C})\]
by the character $|\det m_1(pw_{(i,j)}pw_{(i,j)}^{-1})|^{-5}\delta_{(i,j)}(p)$.  Also, let $\Psi_{s,(i,j)}$ denote the character
\[\Psi_{s,(i,j)}(p) = |\det m_1(pw_{(i,j)}pw_{(i,j)}^{-1})|^{s-\frac{5}{2}}.\]
Finally, let
\[i(s,O_{(i,j)},n)=\dim\hom_{R_{(i,j)}(\mathbf{R})}(\Psi_{s,(i,j)}, \Xi_{(i,j)}^n).\]

Theorem 5.3.2.3 of \cite{warner} shows that
\begin{equation} \label{eqn:warnerbound} \dim\mathbf{T}_s(O_{(i,j)}) \le \sum_{n=0}^\infty i(s,O_{(i,j)},n) \textrm{ and }\dim\mathbf{T}_s(O_{(0,0)}) \le i(s,O_{(0,0)},0).\end{equation}
We will compute these bounds in our case.  Observe that
\[R_{(0,0)} = \diag(m_1,m_2,m_1^*), R_{(1,1)} = \(\begin{array}{ccccc} t_1 & * & * & * & *\\ & t_2 & & & * \\ & & m_2 & & *\\ & & & t_2^{-1} & *\\ & & & & t_1^{-1}\end{array}\), \textrm{ and } R_{(2,2)} = R,\]
where $m_2 \in \Sp_2(\mathbf{R})$.  We have $\Psi_{s,(0,0)}(p)=1$, $\Psi_{s,(1,1)}(p) = |t_1|^{2s-5}$, and $\Psi_{s,(2,2)}(p) = |\det m_1|^{2s-5}$.  Writing $\Lambda_n = \mathrm{Sym}^n\mathfrak{O}_{(i,j)}$ for each case, we also have $\Xi_{(0,0)}^n(p) = \Lambda_n(p)$, $\Xi_{(1,1)}^n(p) = |t_1|^{-4} \Lambda_n(p)$, and $\Xi_{(2,2)}^n(p) = |\det m_1|^{-5} \Lambda_n(p)$.  Since $\mathfrak{O}_{(0,0)}=1$, we have $\Xi_{(0,0)}^n(p) = 1$.

For the remaining Weyl group elements, we have
\[R_{(0,1)} = \(\begin{array}{cccccc} t_1 & * &  & * &  & * \\ & t_2 & & * & &  \\ & & t_3 & * & * & *\\ & & & t_3^{-1} & & \\ & & & & t_2^{-1} & * \\ & & & & & t_1^{-1}\end{array}\) \textrm{ and } R_{(1,2)} =  \(\begin{array}{cccccc} t_1 & * & * & * & * & * \\ & t_2 & & * & * & * \\ & & t_3 & * & * & *\\ & & & t_3^{-1} & & * \\ & & & & t_2^{-1} & * \\ & & & & & t_1^{-1}\end{array}\).\]
It follows that $\Psi_{s,(0,1)}(p) = |t_1t_3|^{s-\frac{5}{2}}$ and $\Psi_{s,(1,2)}(p) = |t_1^2t_2t_3|^{s-\frac{5}{2}}$.  We also have $\Xi_{(0,1)}^n(p) = |t_1t_3|^{-1}\Lambda_n(p)$ and $\Xi_{(1,2)}^n(p) = |t_1^2t_2t_3|^{-2}\Lambda_n(p)$.

We consider the restriction of the representation $\mathfrak{O}_{(i,i)}$ to the Levi of $P_{(i,i)}$.  Note that the bound (\ref{eqn:warnerbound}) for the case $(i,j)=(0,0)$ is always 1, coming from the identity map between trivial characters.  Using the notation for $R_{(i,j)}$ above, the representation $\mathfrak{O}_{(1,1)}$ is $t_1^2 \otimes \mathbf{1}_{\mathrm{Sp}_2} \oplus t_1 \otimes \mathbf{C}_{\mathrm{Sp}_2}^2$, where $\mathbf{1}_{\mathrm{Sp}_2}$ and $\mathbf{C}_{\mathrm{Sp}_2}^2$ are the standard 1 and 2-dimensional representions of $m_2 \in \Sp_2$.  For $\mathfrak{O}_{(2,2)}$, we have
\[\underbrace{\mathrm{Sym}^2(\det m_1\oplus \det m_1)}_{=((\det m_1)^2)^{\oplus 3}} \otimes \mathbf{1}_{\mathrm{Sp}_2} \oplus (\det m_1)^{\oplus 2} \otimes \mathbf{C}_{\mathrm{Sp}_2}^2.\]
For the purposes of computing the intertwining space, the only term of $\Lambda_n$ we need to consider comes from the $n^\textrm{th}$ power of $t_1^2$ or $((\det m_1)^2)^{\oplus 3}$ in the respective cases.  But if $n \ge 1$, there are no such maps with $\mathrm{Re}(s)$ positive, since we would require a solution to $2s-5 = -4-2n$ or $2s-5 = -5-2n$.  So for $s > 0$, the only nonzero $i(s,O_{(i,i)},n)$ have $n=0$ and we have
\[i(s,O_{(2,2)},0) = 0 \textrm{ and } i(s,O_{(1,1)},0) = \begin{cases} 1 & s = \frac{1}{2}\\ 0 & \textrm{ otherwise}. \end{cases}\]

Next, we consider the restriction of the representation $\mathfrak{O}_{(i,i+1)}$ to the Levi of $P_{(i,i+1)}$, which in this case is abelian.  For $i=0$ we get the character $t_1t_3$, and for $i=1$, we get $t_1t_3 \oplus t_1t_2 \oplus t_1^2 \oplus t_3t_2$.  If we assume $\mathrm{Re}(s) > \frac{1}{2}$, it follows in the same way as before that only the $n=0$ term contributes.  (We have $i(\frac{1}{2},O_{(0,1)},1) = 1$, in contrast to the case treated by Jiang.)  In this case, for $\mathrm{Re}(s) > \frac{1}{2}$,
\[i(s,O_{(1,2)},0) = 0 \textrm{ and } i(s,O_{(0,1)},0) = \begin{cases} 1 & s = \frac{3}{2}\\ 0 & \textrm{ otherwise}. \end{cases}\]
Combining these calculations, it follows that for $\mathrm{Re}(s) > \frac{1}{2}$,
\[\dim_\mathbf{C} \hom_G(I_{R,\infty}^\mathrm{sm}(s),I_{R,\infty}^\mathrm{sm}(-s)) \le 1 \textrm{ if }s \ne \frac{3}{2} \textrm{ and }\dim_\mathbf{C} \hom_G\(I_{R,\infty}^\mathrm{sm}\(\frac{3}{2}\),I_{R,\infty}^\mathrm{sm}\(-\frac{3}{2}\)\) \le 2.\]
The remainder of the proof will deal with the exceptional case $s=\frac{3}{2}$.

\subsubsection{Intertwining operators and distributions on $O_{(0,0)}$}

From the bound (\ref{eqn:orbitineq}), it would be sufficient to prove that $\mathbf{T}_{\frac{3}{2}}[O_{(0,0)}] = 0$.  We observe that $\mathbf{T}_{\frac{3}{2}}[O_{(0,0)}] \subseteq \mathbf{T}_{\frac{3}{2}}(O_{(0,0)})$, and it turns out that $\dim_\mathbf{C} \mathbf{T}_{\frac{3}{2}}(O_{(0,0)}) = 1$.  We construct the nontrivial member of this space and show that it cannot be extended to an element of $\mathbf{T}_{\frac{3}{2}}[O_{(0,0)}]$.  Consider the intertwining operator
\[\mathfrak{M}_{(0,0)}(s): I_{R,\infty}^\mathrm{sm}(s) \rightarrow I_{R,\infty}^\mathrm{sm}(-s), \mathfrak{M}_{(0,0)}(s)(f_s)(g) = \int_{U_R(\mathbf{R})}f(w_{(0,0)}ng)dn.\]
The integral is convergent for $\mathrm{Re}(s)$ sufficiently large and has meromorphic continuation.  To simplify notation we write $\zeta(s) = \Gamma(\frac{s}{2})$.  Then the normalizing factor
\[c_{(0,0)}(s) = \frac{\zeta(s-\frac{1}{2})\zeta(2s)\zeta(s-\frac{3}{2})}{\zeta(w+\frac{3}{2})\zeta(2w+1)\zeta(w+\frac{5}{2})}\]
has the property that $c_{(0,0)}^{-1}(s)\mathfrak{M}_{(0,0)}(s)$ is holomorphic for $\mathrm{Re}(s) > 0$ and identifies the normalized spherical sections.

It follows that for $\mathrm{Re}(s) > \frac{1}{2}$, the integral $\mathfrak{M}_{(0,0)}(s)$ has a unique simple pole at $s = \frac{3}{2}$.  For a function $\varphi \in C_c^\infty(G)$, define the canonical projection
\[\mathrm{Pr}_s(\varphi)(g) = \int_{R(\mathbf{R})} |\det m_1|^{-s - \frac{5}{2}}f(pg)dp,\]
where $m_1$ stands for the top left $2 \times 2$ block of $p$ (as defined in (\ref{eqn:levidec})).  Then we can use $\mathfrak{M}_{(0,0)}(s)$ to define a distribution on $\varphi \in C_c^\infty(G)$ by
\begin{equation} \label{eqn:defdist} \widetilde{\mathfrak{M}}_{(0,0)}(s)(\varphi) = \mathfrak{M}_{(0,0)}(s)(\mathrm{Pr}_s(\varphi))(\mathbf{1}_{G(\mathbf{R})}),\end{equation}
where $\mathbf{1}_{G(\mathbf{R})}$ denotes the identity element.  This is not meaningful at the poles, including the value $s = \frac{3}{2}$.  However, the same formula defines $\widetilde{\mathfrak{M}}_{(0,0)}(\frac{3}{2})$ as a distribution on the smaller space $C_c^\infty(\Omega_{(0,0)})$.  In fact, for $p_1,p_2 \in R(\mathbf{R})$, an easy computation yields
\begin{equation} \label{eqn:p1p2equiv} (p_1,p_2)\circ \widetilde{\mathfrak{M}}_{(0,0)}(s) = |\det m_1(p_1p_2)|^{s-\frac{5}{2}}\widetilde{\mathfrak{M}}_{(0,0)}(s)(\varphi),\end{equation}
so $\widetilde{\mathfrak{M}}_{(0,0)}(s) \in \mathbf{T}_s(O_{(0,0)})$ for all $s$.

To obtain our needed dimension bound, it suffices to verify that $\widetilde{\mathfrak{M}}_{(0,0)}(\frac{3}{2})\in \mathbf{T}_{\frac{3}{2}}(O_{(0,0)})$ cannot be extended to an element of $\mathbf{T}_{\frac{3}{2}}[O_{(0,0)}]$.  We use the argument of \cite[\S 4.2.3]{jiang}.  Write
\[\widetilde{\mathfrak{M}}_{(0,0)}(s)(\varphi) = \frac{1}{s-\frac{3}{2}} \mathrm{Res}_{s=\frac{3}{2}}(\widetilde{\mathfrak{M}}_{(0,0)}(s))(\varphi) + \mathrm{PV}_{s=\frac{3}{2}}(\widetilde{\mathfrak{M}}_{(0,0)}(s))(\varphi) + O\(s-\frac{3}{2}\),\]
where $\mathrm{PV}_{s=\frac{3}{2}}(\widetilde{\mathfrak{M}}_{(0,0)}(s))(\varphi)$ denotes the second term in the Taylor expansion of $\widetilde{\mathfrak{M}}_{(0,0)}(s)(\varphi)$ around $s=\frac{3}{2}$.  To ease notation, define $\Lambda^{(-1)}(\frac{3}{2})=\mathrm{Res}_{s=\frac{3}{2}}(\widetilde{\mathfrak{M}}_{(0,0)}(s))$ and $\Lambda^{(0)}(\frac{3}{2}) = \mathrm{PV}_{s=\frac{3}{2}}(\widetilde{\mathfrak{M}}_{(0,0)}(s))$.  Observe that due to the normalized intertwining operator having a simple pole at $s= \frac{3}{2}$, $\Lambda^{(-1)}(\frac{3}{2})$ defines a distribution on $C_c^\infty(G)$, as does $\Lambda^{(0)}(\frac{3}{2})$.

We have a Taylor expansion
\[|\det m_1(p_1p_2)|^{s-\frac{5}{2}} = |\det m_1(p_1p_2)|^{-1} + |\det m_1(p_1p_2)|^{-1}\ln|\det m_1(p_1p_2)|(s-\frac{3}{2}) + O((s-\frac{3}{2})^2)\]
of the factor in (\ref{eqn:p1p2equiv}).  Substituting this expansion into the equivariance equation for $\widetilde{\mathfrak{M}}_{(0,0)}(s)$, we see that
\[(p_1,p_2)\circ \Lambda^{(-1)}\(\frac{3}{2}\) = |\det m_1(p_1p_2)|^{-1} \Lambda^{(-1)}\(\frac{3}{2}\)\]
as distributions and
\begin{equation} \label{eqn:invpv} (p_1,p_2)\circ \Lambda^{(0)}\(\frac{3}{2}\) = |\det m_1(p_1p_2)|^{-1}\Lambda^{(0)}\(\frac{3}{2}\) + \ln|\det m_1(p_1p_2)|^{-1}\Lambda^{(-1)}\(\frac{3}{2}\).\end{equation}
Since $\Lambda^{(-1)}(\frac{3}{2})$ vanishes on $C^\infty_c(\Omega_{(0,0)})$, it defines an element of $\mathbf{T}_{\frac{3}{2}}[O_{(0,1)}]$.  It also follows that $\Lambda^{(0)}(\frac{3}{2})$ defines an extension of $\widetilde{\mathfrak{M}}_{(0,0)}(s)$ from $C^\infty_c(\Omega_{(0,0)})$ to $C_c^\infty(G)$.

Now assume that there exists a nonzero element $\Lambda \in \mathbf{T}_{\frac{3}{2}}[O_{(0,0)}]$, or, after scaling, a distribution $\Lambda \in C^\infty_c(G)^\vee$ that satisfies the invariance property defining $\mathbf{T}_{\frac{3}{2}}$ and restricts to $\widetilde{\mathfrak{M}}_{(0,0)}(\frac{3}{2})$ on $C_c^\infty(\Omega_{(0,0)})$.  We consider $\Lambda' = \Lambda-\Lambda^{(0)}(\frac{3}{2})$; observe that by (\ref{eqn:invpv}), $\Lambda^{(0)}(\frac{3}{2})$ fails to satisfy $(p_1,p_2)$-invariance, so $\Lambda'$ must be nonzero.  Moreover, $\Lambda'$ is a distribution supported on $\ol{O}_{(0,1)}$ with
\begin{equation} \label{eqn:lambdatrans} (p_1,p_2) \circ \Lambda' = |\det m_1(p_1p_2)|^{-1} \Lambda' + \ln|\det m_1(p_1p_2)|^{-1}\Lambda^{(-1)}\(\frac{3}{2}\)\end{equation}
for $p_1,p_2 \in R(\mathbf{R})$.  Note as a special case that for $n_1,n_2$ in the unipotent radical $U_R(\mathbf{R})$, $(p_1,p_2) \circ \Lambda' = \Lambda'$.

\subsubsection{Intertwining operators and distributions on $O_{(0,1)}$}

To proceed further, it will be necessary to explicitly identify the distribution $\Lambda^{(-1)}(\frac{3}{2})$ in order to compute its support.  Notice that $\Lambda^{(-1)}(\frac{3}{2}) \in \mathbf{T}_{\frac{3}{2}}[O_{(0,1)}]$ and $\dim \mathbf{T}_{\frac{3}{2}}[O_{(0,1)}] = 1$, so it suffices to find a nonzero element of this space.

Following Jiang \cite[p.\ 142]{jiang}, we will construct $\Lambda^{(-1)}(\frac{3}{2})$ as the specialization of a composition of two intertwining operators.  Unlike $\mathfrak{M}_{(0,0)}(s)$, this composition will not give a homomorphism $I_{R,\infty}^\mathrm{sm}(s) \rightarrow I_{R,\infty}^\mathrm{sm}(-s)$ for varying $s$, but instead will specialize to one only at $s=\frac{3}{2}$.

We recall that the Levi factor $M_R$ of $R$ has been identified with $\GL_2 \times \Sp_2$.  We refer to the Borel subgroups of the respective factors by $B_{\GL_2}$ and $B_{\Sp_2}$.  We will define an intertwining operator
\[\mathfrak{M}_{(0,1)}(s): I_{R,\infty}^\mathrm{sm}(s) \rightarrow \ind_{R(\mathbf{R})}^{G(\mathbf{R})}\(\left[|\det|^{-\frac{2s+3}{4}} \ind_{B_{\GL_2}(\mathbf{R})}^{\GL_2(\mathbf{R})} |\frac{t_1}{t_2}|^{\frac{2s-1}{4}}\right] \otimes \ind_{B_{\Sp_2}(\mathbf{R})}^{\Sp_2(\mathbf{R})} |t_3|^{s-\frac{1}{2}}\),\]
where the tensor product is between representations of the respective factors $\GL_2$ and $\Sp_2$ of the Levi of $R$ and we have used the notation
\[\diag(t_1,t_2,t_3,t_3^{-1},t_2^{-1},t_1^{-1})\]
for an element of the torus.  The operator $\mathfrak{M}_{(0,1)}(s)$ is defined by the formula
\[\mathfrak{M}_{(0,1)}(s)(f_s) = \int_{U_{w_{(0,1)}}(\mathbf{R})} f_s(w_{(0,1)}ng)dn, \textrm{ where }U_{w_{(0,1)}} = \(\begin{array}{cccccc} 1 &  & * &  & * &  \\ & 1 & * &  & * & * \\ & & 1 &  &  & \\ & & & 1 & * & * \\ & & & & 1 &  \\ & & & & & 1\end{array}\).\]
We next define the intertwining operator
\begin{align*} \mathfrak{M}'(s):&\ind_{R(\mathbf{R})}^{G(\mathbf{R})}\(\left[|\det|^{-\frac{2s+3}{4}} \ind_{B_{\GL_2}(\mathbf{R})}^{\GL_2(\mathbf{R})} |\frac{t_1}{t_2}|^{\frac{2s-1}{4}}\right] \otimes \ind_{B_{\Sp_2}(\mathbf{R})}^{\Sp_2(\mathbf{R})} |t_3|^{s-\frac{1}{2}}\)\\
 \rightarrow &\ind_{R(\mathbf{R})}^{G(\mathbf{R})}\(\left[|\det|^{-\frac{2s+3}{4}} \ind_{B_{\GL_2}(\mathbf{R})}^{\GL_2(\mathbf{R})} |\frac{t_1}{t_2}|^{-\frac{2s-1}{4}}\right] \otimes \ind_{B_{\Sp_2}(\mathbf{R})}^{\Sp_2(\mathbf{R})} |t_3|^{-s+\frac{1}{2}}\)
\end{align*}
to be the composite of the usual nontrivial intertwining operators on $\GL_2$ and $\Sp_2$.  More precisely, let $w' = \diag(\mathbf{J}_2,\mathbf{J}_2,\mathbf{J}_2)$ in block diagonal form and define
\[\mathfrak{M}'(s)(f_s) = \int_{U'(\mathbf{R})} f_s(w'ng)dn, \textrm{ where }U' = \diag(\mm{1}{*}{0}{1},\mm{1}{*}{0}{1},\mm{1}{*}{0}{1}).\]
The normalizing factor for the composition $\mathfrak{M}_{(0,1)}'(s) = \mathfrak{M}'(s) \circ \mathfrak{M}_{(0,1)}(s)$ is given by
\[c(s)=\frac{\zeta(s-\frac{1}{2})^2\zeta(2s)}{\zeta(s+\frac{5}{2})\zeta(s+\frac{3}{2})\zeta(2s+1)}.\]
Notice that this factor is holomorphic at $s=\frac{3}{2}$ and that the range of $\mathfrak{M}_{(0,1)}'(\frac{3}{2})$ is $I_{R,\infty}^\mathrm{sm}(-\frac{3}{2})$.  Due to this holomorphy and using the definition (\ref{eqn:defdist}), $\mathfrak{M}_{(0,1)}'(\frac{3}{2})$ defines a non-zero distribution $\widetilde{\mathfrak{M}}_{(0,1)}'(\frac{3}{2})$ on all of $C_c^\infty(G)$.

\subsubsection{Calculus of distributions}

The form of the defining integral implies that $\widetilde{\mathfrak{M}}_{(0,1)}'(\frac{3}{2})$ has support equal to $\ol{O}_{(0,1)}$.  Moreover, we have
\[(p_1,p_2) \circ \widetilde{\mathfrak{M}}_{(0,1)}'(\frac{3}{2}) = |\det m_1(p_1p_2)|^{-1} \widetilde{\mathfrak{M}}_{(0,1)}'(\frac{3}{2}),\]
so $\widetilde{\mathfrak{M}}_{(0,1)}'(\frac{3}{2}) \in \mathbf{T}_{\frac{3}{2}}[O_{(0,1)}]$.  Recall that this space is at most 1-dimensional by our bound above.  It follows that $\Lambda^{(-1)}(\frac{3}{2})$ is proportional to $\widetilde{\mathfrak{M}}_{(0,1)}'(\frac{3}{2})$.  In particular, the support of $\Lambda^{(-1)}(\frac{3}{2})$ is equal to $\ol{O}_{(0,1)}$.  It is then easy to see using (\ref{eqn:lambdatrans}) that $\Lambda'$ has support equal to $\ol{O}_{(0,1)}$.

We now show that a $\Lambda' \in C_c^\infty(G)^\vee$ that satisfies (\ref{eqn:lambdatrans}) and has support equal to $\ol{O}_{(0,1)}$ cannot exist.  We can restrict both $\Lambda'$ and $\Lambda^{(-1)}(\frac{3}{2})$ to $C_c^\infty(\Omega_{(0,1)})$, and again to $C_c^\infty(O_{(0,1)})$, since they are supported away from $\Omega_{(0,0)} = \Omega_{(0,1)} \setminus O_{(0,1)}$.  Both are still nonzero because their support was originally $\ol{O}_{(0,1)}$.  Write $B$, $T$, and $U_B$ for the upper-triangular Borel of $G$, diagonal torus of $G$, and the unipotent radical of $B$, respectively.  We finally restrict again to $C_c^\infty(\Omega)$ for the open subset $\Omega \subseteq O_{(0,1)}$ given by $\Omega = U_B(\mathbf{R})T(\mathbf{R})w_\ell U_{w_\ell}(\mathbf{R})$, where $w_\ell = w'w_{(0,1)} w^{\prime -1}$ is the longest element in the double coset $W_Rw_{(0,1)}W_R$, $w' = \diag(\mathbf{J}_2,\mathbf{J}_2,\mathbf{J}_2)$ as above, and
\[U_{w_\ell} = \(\begin{array}{cccc} 1 & * & * & * \\ &  \mathbf{1}_2 & * & * \\ & & \mathbf{1}_2 & *\\ & & & 1\end{array}\).\]

Note that (\ref{eqn:lambdatrans}) implies that $\Lambda'$ is not $(p_1,p_2)$-equivariant, while $\Lambda^{(-1)}(\frac{3}{2})$ is, so to derive a contradiction it would suffice to show that $\Lambda^{(-1)}(\frac{3}{2}) = -\Lambda'$.  Define by $L_X$ or $R_X$, respectively, the action of an element $X$ of the Lie algebra $\mathfrak{g}$ of $G(\mathbf{R})$ by differentiation on functions on the left or right of $C_c^\infty(\Omega)^\vee$.  Let $\mathfrak{t}$ be the Lie algebra of $T(\mathbf{R})$; we identify $\mathfrak{t}$ with $\mathbf{R}^3$ in such a way that the exponential map is
\[\mathbf{R}^3 \ni (a_1,a_2,a_3) \mapsto \diag(\exp(a_1),\exp(a_2),\exp(a_3),\exp(-a_3),\exp(-a_2),\exp(-a_1)).\]
For all $a,b \in \mathbf{R}$, $(0,a,b) \in \mathfrak{t}$ satisfies $L_{(0,a,b)} = -aT-a\Lambda^{(-1)}(\frac{3}{2}) = R_{(0,a,b)}$ by (\ref{eqn:lambdatrans}), so we need only show that $L_{(0,1,0)} = R_{(0,1,0)} = 0$ to prove the result.  The property (\ref{eqn:lambdatrans}) also implies $L_X \ast \Lambda' = 0 = R_X \ast \Lambda'$ for any $X \in \mathfrak{u}_B$, the Lie algebra of $U_B(\mathbf{R})$, as well as $L_{(a,-a,0)} \ast \Lambda' = 0 = R_{(a,-a,0)} \ast \Lambda'$ and $L_{(0,0,a)} \ast \Lambda' = 0 = R_{(0,0,a)} \ast \Lambda'$.

We can compute the action of $X=(0,a,b)$ on the left and right in terms of the decomposition $\Omega=U_B(\mathbf{R})T(\mathbf{R})w_\ell U_{w_\ell}(\mathbf{R})$.  Note that this decomposition gives an isomorphism $U_B(\mathbf{R}) \times T(\mathbf{R}) \times U_{w_\ell}(\mathbf{R}) \rightarrow \Omega$ of real manifolds.  So $\varphi \in C_c^\infty(\Omega)$ can be written as $\varphi(n,t,n')$.  By definition, the right action $L_X$ on $\Lambda'$ is given by
\[(L_X \ast \Lambda')(\varphi) = (\Lambda')(X \ast \varphi) = (\Lambda')(ntw_\ell n' \mapsto \frac{d}{dt}\varphi(\exp(tX)ntw_\ell n').\]
We observe that
\[\exp(tX)ntw_\ell n' = (\exp(tX)n\exp(-tX))\exp(tX)tw_\ell n',\]
where since $T$ normalizes $U_B$, the right-hand side now corresponds to $\varphi(\mathrm{ad}(tX)(n),\exp(tX)t,n')$.  In other words,
\[(L_X \ast \Lambda')(\varphi) = \Lambda'(L_{X'} \ast \varphi + (at_2 \frac{\partial}{\partial t_2} + bt_3\frac{\partial}{\partial t_3})(\varphi))=\Lambda'(at_2 \frac{\partial \varphi}{\partial t_2} + bt_3\frac{\partial \varphi}{\partial t_3})\]
for some $X' \in \mathfrak{u}_B$.  We similarly use
\[ntw_\ell n'\exp(-tX) = n (tw_\ell\exp(-tX)w_\ell^{-1}) w_\ell (\exp(tX)n'\exp(-tX))\]
to obtain
\[(R_X \ast \Lambda')(\varphi) = \Lambda'(R_{X''} \ast \varphi + (bt_2 \frac{\partial}{\partial t_2} + at_3\frac{\partial}{\partial t_3})(\varphi))=\Lambda'(-bt_2 \frac{\partial\varphi}{\partial t_2} -at_3\frac{\partial\varphi}{\partial t_3})\]
for $X'' \in \mathfrak{u}_{w_\ell}$, the Lie algebra of $U_{w_\ell}(\mathbf{R})$.  It follows from varying $a$ and $b$ in $L_X \ast \Lambda' = R_X \ast \Lambda'$ that $L_X \ast \Lambda' = R_X \ast \Lambda'=0$ for all $a,b$.  Setting $a=1$ and $b=0$ gives the identity we need.

\subsection{Residues of the two-variable Eisenstein series} \label{subsec:res}

In the remainder of the section, we only consider poles of $L$-functions to the right of the line with real part $\frac{1}{2}$.  We also shift the variable $w$ of the Eisenstein series by 2 so that the functional equation is given by $w \mapsto 1-w$ instead of $w \mapsto 5-w$.  Recall that we have used $P$ to denote the non-maximal parabolic stabilizing the flag $\gen{f_1,f_2} \subseteq \gen{f_1,f_2,f_3}$.  We can write the section of the two variable Eisenstein series as lying in an iterated induction.  Write $R$ for the maximal parabolic with Levi factor $\GL_2 \times \GSp_2$.  Here, $\GSp_2$ refers to the group of matrices of the form $\diag(\mathbf{1}_2,m,\mathbf{1}_2)$ and $B_{\GSp_2}$ is the upper triangular Borel subgroup of $\GSp_2$.  Recall that our section is an element $f_{s,w} \in \ind_{P(\mathbf{A})}^{G(\mathbf{A})} \chi_{s,w}$, where $\chi_{w,s}(m_1,m_2,\mu) = |\det m_1|_{\mathbf{A}}^{w+2} |m_2|_{\mathbf{A}}^{2s}|\mu|_{\mathbf{A}}^{-s-w-2}$ in the notation of (\ref{eqn:defp}).  We regard $f_{s,w}$ as being in the iterated induction
\[f_{s,w} \in \ind_{R(\mathbf{A})}^{G(\mathbf{A})} \(|\det m_1|_{\mathbf{A}}^{w+2} \cdot \(\ind_{P(\mathbf{A})}^{R(\mathbf{A})} |\mu|_{\mathbf{A}}^{-s-w-2}|m_2|_{\mathbf{A}}^{2s}\)\).\]
We may identify the inner induction with $\ind_{B_{\GSp_2}(\mathbf{A})}^{\GSp_2(\mathbf{A})} |\mu|_{\mathbf{A}}^{-s-w-2}|m_2|_{\mathbf{A}}^{2s}$.

Denote by $\chi'_w: R(\A) \rightarrow \C^\times$ the unique character whose restriction to $P(\A)$ is $\chi_{w, 0}$, the character $\chi_{w,s}$ evaluated at $s=0$.  For $f \in \ind_{P(\A)}^{\GSp_6(\A)}(\chi_{w,s})$, set 
\[r(f)(g,w) = \mathrm{res}_{s=1} \sum_{\gamma \in P(\Q)\backslash R(\Q)}{f(\gamma g, s,w)}.\]
One has the following lemma.
\begin{lemma}\label{s=1:lemma} For $f$ and $r(f)$ as above, one has $r(f) \in \ind_{R(\A)}^{\GSp_6(\A)}(\chi'_w)$.  Furthermore, the residue at $s=1$ of the Eisenstein series $E(g,s,w) = \sum_{\gamma \in R(\Q)\backslash \GSp_6(\Q)}{f(\gamma g,s,w)}$ is $E(g,w,r(f)) = \sum_{\gamma \in R(\Q) \backslash \GSp_6(\Q)}{r(f)(\gamma g,w)}$. \end{lemma}
\begin{proof} This follows from the fact that the residues of degenerate Eisenstein series on $\GL_2$ are constant.\end{proof}

Let $R' = R \cap Sp_6$ and let $Q \subseteq \Sp_6$ be the parabolic stabilizing the line spanned by $f_3$ in $V$.  For a standard section $\varphi_{s'} \in \ind_{Q(\A)}^{\Sp_6(\A)}(\delta_{Q}^{s'/6+ 1/2})$, let $E_{Q}(g,\varphi_{s'})$ denote the associated Eisenstein series.  We have the following identity.
\begin{proposition}\label{FTQ} Let $f_{s,w}$ denote a standard section in $\ind_{P(\A)}^{\GSp_6(\A)}(\chi_{s,w})$, and denote by $E(g,f_{s,w})$ the associated Eisenstein series.  There is a standard section $\varphi_{s'}$ as above so that one has the equality
\[\mathrm{Res}_{w=2} \mathrm{Res}_{s=1} E(g,f_{s,w}) = E_{Q}(g,\varphi_{s'})|_{s'=3/2}\]
for all $g \in \Sp_6(\A)$. \end{proposition}
\begin{proof} The residue at $s=1$ of $E(g,f_{s,w})$ is computed in Lemma \ref{s=1:lemma}.  The residue at $w=2$ of the resulting function is given by Jiang's first term identity -- the case $n=3, r=2, \ell=1$ of Proposition 3.1 of \cite{jiang2}. Note that this proposition makes the assumption that \cite[Assumption/Proposition]{jiang2} holds; we verify this in Proposition \ref{prop:assume} above.
\end{proof}

\subsection{Poles of $L$-functions} \label{subsec:poles}

We prove the following result.
\begin{prop} \label{prop:poles}
Let $S$ be a set of primes consisting of the Archimedean place as well as all places where $\pi_1$ or $\pi_2$ is ramified.  Suppose that $L^S(w,\pi_2,\mathrm{Std})$ has a pole at $w=2$.  Then $L^S(s,\pi,\mathrm{Std}\times\mathrm{Spin})$ cannot have a pole at $s=1$.
\end{prop}

We begin with the proof of the following result, which controls the local integrals $I(W_1,W_2,s,w)$ at the Archimedean place, and at bad finite places.
\begin{prop} \label{prop:ramcalc}
For any finite place $v$ and any $\pi$ on $\GL_2 \boxtimes \GSp_4$ with trivial central character, the data of $\phi_{1,v},\phi_{2,v},$ and $f_v(g,s,w)$ can be selected to make the local integral $I(W_1, W_2,s,w)$ defined by (\ref{eqn:localint}) constant and nonzero.  If $v$ is Archimedean, the local integral $I(W_1, W_2,s,w)$ has meromorphic continuation in $s$ and $w$.  If $s_0,w_0 \in \mathbf{C}$ are given, it is possible to choose $\phi_{1,v},\phi_{2,v},$ and $f_v(g,s,w)$ so that $I(W_1, W_2,s,w)$ is non-zero at $(s_0,w_0)$.
\end{prop}
\begin{proof} For ease of notation, set $S_5 = \mathrm{Stab}_5 \subseteq H$, where $\mathrm{Stab}_5$ is as defined in the proof of Theorem \ref{unfolding}.  Since we are working locally, we also write only the name of a group for its points over the local field $\mathbf{Q}_v$.  We have $E(f_3)'E(f_2) = U_{B_H} \cap S_5$, where $U_{B_H}$ denotes the unipotent radical of the upper triangular Borel subgroup of $H$.  Thus the local integral to be computed is
\[I(W,f,s,w) = \int_{(U_{B_H} \cap S_5)Z \backslash H}{W(g)f'(g,s,w)\,dg}\]
where $W(g) = W_1(g_1)W_2(g_2)$ denotes the Whittaker function on $H$.

Denote by $S_5' \subseteq S_5$ the subgroup of $S_5$ consisting of those $g = (g_1,g_2) \in \GL_2 \boxtimes \GSp_4$ with $f_2 g_2 = f_2$.  In matrices, $S_5'$ consists of elements of the form
\begin{equation}\label{S'eqn} g=(g_1,g_2) = \left( \left(\begin{array}{cc} a& -b \\ -c&d\end{array}\right), \left(\begin{array}{cccc}\delta &&& * \\&a&b&\\&c&d& \\ &&&1\end{array}\right)\right),\end{equation}
where $\delta = ad -bc =\mu(g)$.  Denote by $L_i$ a non-zero Whittaker functional on $\pi_{i,v}$ for $i \in \set{1,2}$.  The proof consists of two parts.  The first is an analysis of the integral
\[I'(v_1,v_2,s) = \int_{S_5 \cap U_{B_H} \backslash S_5'}{|\mu(g)|^{s} L_1(g_1 v_1) L_2(g_2 v_2)\,dg},\]
and the second will reduce the study of the integral $I(W,f,s,w)$ to that of $I'(v_1,v_2,s)$.

Our analysis of the integral $I'(v_1,v_2,s)$ is given by the following lemma.
\begin{lemma}\label{I'lem} If $v$ is finite, there exists data $v_1, v_2$ so that $I'(v_1,v_2,s)$ is a nonzero constant independent of $s$.  If the place $v$ is Archimedean, then the integral $I'(v_1,v_2,s)$ has meromorphic continuation in $s$, and this meromorphic continuation is continuous in the choice of vectors $v_1, v_2$.  Furthermore, with appropriate data, it can be made nonzero at any fixed $s_0$. \end{lemma}  
\begin{proof} We deduce the Archimedean part of the lemma from an analogous result for Novodvorsky's $\GL_2 \times \GSp_4$ integral \cite{novodvorsky}.  For the finite places, we prove the statement of the lemma directly.

Suppose first that $\Q_v = \R$.  The local integral considered in \cite{novodvorsky} is
\[J(v_1,v_2,s) = \int_{U_2 \times U_2 \backslash \GL_2 \boxtimes \GL_2}{\Phi_{2}((0,1)g_2)|\det(g)|^{s}L_1'(g_1 v_1)L_2(g v_2)\,dg}.\]
Here, $g = (g_1,g_2)$ is an element of $\GL_2 \boxtimes \GL_2$, which is embedded in $\GSp_4$ by choosing a splitting of the four-dimensional symplectic space $W_4$ into a direct sum $W_2 \oplus W_2$ of two-dimensional symplectic subspaces, $\Phi_2$ denotes a Schwartz-Bruhat function on the second $W_2$ factor, $U_2$ is the unipotent radical of the upper-triangular Borel subgroup of each $\GL_2$, and $L_1'$ denotes the Whittaker functional of $\pi_{1,v}$ on $\GL_2$ that transforms by the character $\mm{1}{x}{}{1} \mapsto \psi(-x)$.

Set $\epsilon = \mm{-1}{}{}{1}$.  By uniqueness of the Whittaker functional, we have $L_1'(v_1) = L_1(\epsilon v_1).$  For $h \in \GL_2$, set $j(h) = \epsilon h \epsilon^{-1}$.  Then $j(\mm{a}{b}{c}{d}) = \mm{a}{-b}{-c}{d}$.  It follows that
\begin{align}\label{Jepint} \nonumber J(\epsilon^{-1}v_1, v_2,s) &= \int_{U_2 \times U_2 \backslash \GL_2 \boxtimes \GL_2}{\Phi_{2}((0,1)g_2)|\det(g)|^{s}L_1(j(g_1) v_1)L_2(g v_2)\,dg} \\ \nonumber &= \int_{B_2' \backslash \GL_2}{\Phi_{2}((0,1)g_2)\int_{ U_2 \times U_2 \backslash \GL_2 \boxtimes B_2'}{|\mu(r)|^{s}L_1(j(r_1)v_1)L_2(r gv_2)\,dr}\,dg} \\ &= \int_{B_2' \backslash \GL_2}{\Phi_{2}((0,1)g_2)I'(v_1,gv_2,s)\,dg_2}.\end{align}
Here $B_2' \subseteq \GL_2$ is the subgroup consisting of matrices of the form $\mm{*}{*}{0}{1}$.

Now we apply the well-known result of Dixmier-Malliavin \cite{dixmierMalliavin} to the group $\SL_2(\Q_v)$ acting on the vector $v_2$, where $\SL_2$ is embedded in $\Sp_4$ via its action on the vectors $e_2, f_2$.  In particular, we have that $v_2 = \sum_{j}{\varphi_j * v_{2,j}}$, where the sum is finite, the $v_{2,j}$ are smooth vectors in the space of $\pi_{2,v}$, and $\varphi_j$ is in $C^{\infty}_{c}(\SL_2(\Q_v))$.  We obtain that
\begin{equation} \label{eqn:iprimesum} I'(v_1,v_2,s) = \sum_{j} \int_{U_2 \backslash \SL_2}{\xi_j(g) I'(v_1, gv_{2,j},s)\,dg},\end{equation}
where
\[\xi_j(g) = \int_{\Q_v}{\varphi_j(\mm{1}{u}{}{1}g)\,du}.\]
But the functions $\xi_j(g)$ are easily seen to be of the form $\Phi_j((0,1)g)$ for compactly supported Schwartz-Bruhat functions $\Phi_j$ on the $1 \times 2$ row vectors $W_2$.  Hence the integral $I'(v_1,v_2,s)$ can be written as a finite sum of integrals of the form $J(v_1, v_2',s)$.  We conclude that the integrals $I'(v_1,v_2,s)$ have meromorphic continuation, which is continuous in the parameters, since the same is true of the integrals $J(v_1,v_2,s)$ by \cite[Theorem A]{soudryArch}.  Furthermore, if this meromorphic continuation vanishes at some fixed $s = s_0$ for all choices of data $v_1, v_2$, then by (\ref{Jepint}) the same vanishing would hold for $J(v_1,v_2,s)$.  However, by \cite[Proposition, \S 7.2]{soudryRSAxB} and the same non-vanishing statement for the classical Rankin-Selberg convolution on $\GL_2 \times \GL_2$ \cite[\S 17, \S 18]{jacquetRS}, this latter integral can be made nonvanishing at $s_0$.

Now assume that $v$ is a finite place.  Note that if $g =(g_1,g_2) \in S_5'$ has $g_1 = \mm{a}{*}{0}{1}$, then $L_1(g_1 u_1(x) v_1) = \psi(ax)L_1(g_1 v_1)$.  Now suppose $\xi_1$ is a Schwartz-Bruhat function on $\mathbf{G}_a$, and define
\[v_1' = \int_{F}{\xi_1(z)u_1(z) \cdot v_1\,dz}.\]
Then $L_1(g v_1') = \widehat{\xi}(a) L_1(g v_1)$, where $\widehat{\xi_1}$ is the Fourier transform of $\xi_1$, and we still assume $g = (g_1,g_2)$ has $g_1  = \mm{a}{*}{0}{1}$.

Next, note that if $g \in S_5'$, then in the notation of (\ref{S'eqn}), $L_2(  g u_2(x,y) v_2) = \psi(-cy+dx)L_2(g v_2).$  Hence, if $\xi_2$ is a Schwartz-Bruhat function on $\mathbf{G}_a^2$, and
\[v_2' = \int_{F^2}{\xi_2(x,y) u_2(x,y) \cdot v_2\,dx\,dy},\]
 then $L_2( g \cdot v_2') = \widehat{\xi_2}(c,d) L_2 (g \cdot v_2)$, where $\widehat{\xi_2}$ is an appropriate Fourier transform.

Choose $v_1, v_2$ so that $L_1(v_1) \neq 0$ and $L_2(v_2) \neq 0$.  Choose $\xi_1, \xi_2$ so that $\widehat{\xi_1}$ is nonnegative and supported very close to $1$, and $\widehat{\xi_2}$ is nonnegative and supported very close to $(0,1)$.  Let $v_1', v_2'$ be as above.  Denote by $S_5'' \subseteq S_5'$ the subgroup consisting of those $g=(g_1,g_2) \in S_5'$ with $g_1$ of the form $\mm{*}{*}{0}{1}$.  Then we have
\begin{align*}
  I'(v_1',v_2',s) &= \int_{S_5 \cap U_{B_H} \backslash S_5'}{|\mu(g)|^{s} \widehat{\xi_2}(c,d) L_1(g_1 v_1') L_2(g_2 v_2)\,dg}\\
  &= C \int_{S_5 \cap U_{B_H} \backslash S_5''}{|\mu(g)|^{s} L_1(g_1 v_1') L_2(g_2 v_2)\,dg} \\
	&=  C \int_{S_5 \cap U_{B_H} \backslash S_5''}{|\mu(g)|^{s} \widehat{\xi_1}(a) L_1(g_1 v_1) L_2(g_2 v_2)\,dg} \\
	&= C' L_1(v_1) L_2(v_2).
\end{align*}
Here $C, C'$ are positive constants that come from the measures of compact open subgroups.  It follows that when the place $v$ is finite, we can choose data $v_1', v_2'$ so that $I'(v_1',v_2',s)$ is constant nonzero independent of $s$.  This completes the proof of the lemma.
\end{proof}

We now consider again the entire local integral $I(W,f,s,w)$.  To reduce the study of this integral to that of $I'(v_1,v_2,s)$, we will carefully choose the Eisenstein section $f(g,s,w)$.  

First, note that there is a surjective map $\wedge^3 W_6 \otimes \mu^{-1} \rightarrow W_6$ of representations of $\GSp_6$; denote its kernel by $W_{14}$.  As a representation of $H = \GL_2 \boxtimes \GSp_4$, $W_{14} = W_4 \oplus W_2 \otimes V_5$, where $V_5 \subseteq \wedge^2 W_4 \otimes \mu^{-1}$ is the $5$-dimensional irreducible representation of $\PGSp_4$.  Moreover, $f_1 \wedge f_2 \wedge f_3 \in W_{14}$.  If $\Phi_{14}$ is a Schwartz-Bruhat function on $W_{14}$, and $g \in \GSp_6(F)$, set
\[f_{(3,3)}(g,s,\Phi_{14}) = |\mu(g)|^{s} \int_{\GL_1(F)}{\Phi_{14}(t f_1 \wedge f_2 \wedge f_3 g)|t|^{2s}\,dt}.\]
Then $f_{(3,3)}(g,s,\Phi_{14})$ is a local section for a degenerate Siegel Eisenstein series on $\GSp_6$.

Similarly, suppose $\Phi_{12}$ is a Schwartz-Bruhat function on $V_{\GL_2} \otimes W_6$, a $12$-dimensional representation of $\GL_2 \times \GSp_6$, where $V_{\GL_2}$ is a space of column vectors on which $m \in \GL_2$ acts by left-multiplication by $m^{-1}$.  We identify $V_{\GL_2} \otimes W_6$ with $M_{2,6}(F)$, the $2 \times 6$ matrices over $F$, by letting $m \in \GL_2$ act by left multiplication by $m^{-1}$ and letting $g \in \GSp_6$ act by right multiplication.  Set
\[f_{(2,2,2)}(g,w-2s,\Phi_{12}) = |\mu(g)|^{w-2s} \int_{\GL_2(F)}{\Phi_{12}\left(h(0|0|1_2)g\right)|\det(h)|^{w-2s}\,dh}.\]
Then $f_{(2,2,2)}(g,w-2s,\Phi_{12})$ is a local section for a degenerate Eisenstein series associated to the parabolic $R$ on $\GSp_6$.  

For Schwartz-Bruhat functions as above, we set $f(g,s,w) = f_{(3,3)}(g,s,\Phi_{14})f_{(2,2,2)}(g,w-2s,\Phi_{12})$.  With this section, the local integral $I(W,f',s,w)$ is readily seen to become
\[\int_{(U_{B_H} \cap S_5) \backslash H \times \GL_2}{|\mu(g)|^{w-s}|\det(h)|^{w-2s}\Phi_{14}(f_1 \wedge f_2 \wedge f_3 \gamma_5 g)\Phi_{12}(h(0|0|1_2)\gamma_5 g)W(g)\,dg\,dh}.\]

Let $V_{26}= W_{14} \oplus W_{12}$ and let $\Phi_{26} = \Phi_{14}\otimes \Phi_{12}$ be a Schwartz-Bruhat function on $V_{26}$.  We define 
\[v_0 = (f_1 \wedge f_2 \wedge f_3 \gamma_5, (0|0|1_2)\gamma_5) \in V_{26}.\]
Let $\GL_2 \times H$ act on $V_{26}$ via $v = (v_{14}, v_{12}) \cdot (m,g) = (v_{14}g, m^{-1} v_{12} g)$.  Here we have $v_{14} \in W_{14}$, $v_{12} \in M_{2,6}$, $m \in \GL_2$, and $g \in H$.
\begin{claim}\label{claim:stab} The stabilizer of $v_0$ inside $\GL_2 \times H$ is $\{(m,g): g = (g_1,g_2) \in S_5', m = g_1\}$. \end{claim}
\begin{proof} One computes
\[f_1 \wedge f_2 \wedge f_3 \gamma_5 = (\omega_2-\omega_1) \wedge f_2 + (e_3 \wedge f_1 + f_3 \wedge e_1) \wedge f_2\]
where $\omega_1 = e_1 \wedge f_1$, $\omega_2 = e_2 \wedge f_2 + e_3 \wedge f_3$ are the symplectic forms of the embedded $\GL_2$ and $\GSp_4$.  Additionally, one has
\begin{equation}\label{phi12} \left(\begin{array}{cc|cc|cc} 0&0 &0 &0 &1&0 \\ 0&0&0&0&0&1\end{array}\right)\gamma_5 = \left(\begin{array}{cc|cc|cc} 1&0 &-1 &0 &0&0 \\ 0&0&0&1&0&1\end{array}\right).\end{equation}
Using the outer columns of the right-hand side of (\ref{phi12}), one concludes that if $(m,g)$ stabilizes $v_0$, then $m = g_1$.  Moreover, the middle columns of (\ref{phi12}) imply that the action of $g_2$ on $\gen{e_3, f_3}$ must be by the matrix $j(m)=j(g_1)$, where we recall that $j$ is the map $\mm{a}{b}{c}{d} \mapsto \mm{a}{-b}{-c}{d}$ on $2 \times 2$ matrices.  The remainder of the computation is straightforward. \end{proof}

Applying Claim \ref{claim:stab}, we deduce that
\begin{align}\label{IYint} \nonumber &I(W,f',s,w) = \int_{(U_{B_H} \cap S_5) \backslash \GL_2 \times H}{|\mu(g)|^{w-s}|\det(m)|^{2s-w}\Phi_{26}(v_0(m,g)) W(g)\,dg\,dm} \\  &= \int_{ S_5'\backslash \GL_2 \times H}{ |\mu(g) \det(m)^{-1}|^{w-2s}\Phi_{26}(v_0(m,g))|\mu(g)|^s \int_{S_5 \cap U_{B_H} \backslash S_5'}{|\mu(r)|^{s}W(rg)\,dr}\,dg\,dm}.\end{align}

Suppose the place $v$ is finite, and denote by $K$ an open compact subgroup of $H$ sufficiently small so that $W$ is right-invariant by $K$.  Denote by $K'$ a small open compact subgroup of $\GL_2$.  It is not difficult to see that we can choose non-negative, compactly supported Schwartz-Bruhat data $\Phi_{14}, \Phi_{12}$ so that $\Phi_{26}(v_0(m,g)) \neq 0$ implies $m^{-1}g_1 \in K'$ and $g \in S_5' K$.  With such data, it follows that $I(W,f',s,w) = C I'(v_1,v_2,s)$ for some positive constant $C$.  Hence by Lemma \ref{I'lem}, we have proved the proposition in the case that the place $v$ is finite.

Consequently, since the global Rankin-Selberg convolution has meromorphic continuation in $s,w$, and we already know that the partial $L$-functions $L^S(\pi,\mathbf{1} \times \mathrm{Std},w)$, $L^{S}(\pi,\mathrm{Std} \times \mathrm{Spin},s)$ have meromorphic continuation, we deduce that the Archimedean integral $I(W,f',s,w)$ has meromorphic continuation in $s,w$.  Now suppose that the meromorphic continuation of $I(W,f',s,w)$ to $s=s_0, w= w_0$ is $0$ for all choices of data $v_1, v_2$.  Then by (\ref{IYint}), since $\Phi_{26}$ can be an arbitrary smooth compactly supported function, we obtain that $I'(v_1,v_2,s)$ vanishes at $s=s_0$ for all $v_1,v_2$.  But by Lemma \ref{I'lem}, this cannot happen.  Hence the proposition is proved. \end{proof}

Using this, we may now give the application of the integral representation to the poles of the $L$-functions.
\begin{proof}[Proof of Proposition \ref{prop:poles}]
Assume that $L^S(w,\pi_2,\mathrm{Std})$ has a pole at $w=2$ and $L^S(s,\pi,\mathrm{Std}\times\mathrm{Spin})$ has a pole at $s=1$.  We will derive a contradiction.

Using the case $s_0=1$, $w_0=2$ of Proposition \ref{prop:ramcalc}, there exist $\phi_1\in \pi_1,\phi_2\in \pi_2$, and a section $f_{s,w}(g)$ for the Eisenstein series $E(g,s,w,f_{s,w})$ so that the integral
\[I(g,\phi_1,\phi_2,s,w,f_{s,w}) =_{(s_0,w_0)} L^S(w,\pi_2,\mathrm{Std})L^S(s,\pi,\mathrm{Std}\times\mathrm{Spin}),\]
where $=_{(*)}$ means that the quotient of the two sides is non-vanishing and holomorphic near the value $(s_0,w_0)$.  (Note that the normalizing factors have no poles or zeroes at $(1,2)$.)  Then by hypothesis, $I(g,\phi_1,\phi_2,s,w,f_{s,w})$ has poles along $s=1$ and $w=2$.

After forming an outer integration over the similitude, we have
\[\int_{ \mathbf{A}^{\times 2}\mathbf{Q}^\times\backslash \mathbf{A}^\times}\int_{(\SL_2\boxtimes \Sp_4)(\mathbf{Q})Z(\mathbf{A}) \backslash (\SL_2\boxtimes \Sp_4)(\mathbf{A})} E(gt(a),w,f_{s,w})\phi_1(g_1t(a))\phi_2(g_2t(a))dg,\]
where $t = \diag(1,1,1,a,a,a)$.  If the residue of the inner integral at $s=1$ and $w=2$ is zero for all values of $t \in \mathbf{A}^\times$, it follows that $L^S(w,\pi_2,\mathrm{Std})$ cannot have a pole at $w=2$ while $L^S(s,\pi,\mathrm{Std}\times\mathrm{Spin})$ has a pole at $s=1$.  Thus, there is such a $t$ for which the inner integral has a nonzero residue at $s=1$ and $w=2$.  By translating all of the data of $\GSp_6(\A)$ and $H(\A)$ by $t$, we deduce that there is a standard section $f'_{s,w}$ and cusp forms $\phi_1', \phi_2'$ so that the integral
\[\int_{(\SL_2\boxtimes \Sp_4)(\mathbf{Q})\backslash (\SL_2\boxtimes \Sp_4)(\mathbf{A})} E(g,w,f'_{s,w})\phi_1'(g_1)\phi_2'(g_2)dg\]
has a nonzero residue at $s=1$ and $w=2$.

Now we apply Proposition \ref{FTQ}.  Applying this proposition, we see that to conclude the proof, it suffices to check that
\[\int_{(\GL_2\boxtimes \GSp_4)(\mathbf{Q})Z(\mathbf{A}) \backslash (\GL_2\boxtimes \GSp_4)(\mathbf{A})} E_Q(g,s',\varphi_{s'})\phi_1(g_1)\phi_2(g_2)dg\]
vanishes identically for all $\varphi_{s'}$.  It is easy to see that the action of $(\GL_2\boxtimes \GSp_4)(\mathbf{Q})$ on lines has three orbits represented by $\gen{f_1}$, $\gen{f_2}$, and $\gen{f_1+f_2}$.  The terms of the unfolded integral corresponding to $\gen{f_1}$ and $\gen{f_2}$ vanish by cuspidality of $\pi_2$.  For $\gen{f_1+f_2}$, observe that the stabilizer includes the unipotent radical of the Klingen parabolic of $\GSp_4$ stabilizing $\gen{f_2}$, so the integral corresponding to this orbit also vanishes by cuspidality.
\end{proof}

\subsection{The integral of Bump-Friedberg-Ginzburg} \label{subsec:nopole}

In this section, we explain why Proposition \ref{prop:poles} can be deduced from the multivariate integral of Bump-Friedberg-Ginzburg \cite{bfg2}.
\begin{prop} \label{prop:nopole}
Suppose that $\pi$ is a generic cuspidal automorphic representation of $\GSp_4(\mathbf{A})$.  Then for any finite set $S$, the Standard $L$-function $L^S(\pi,\mathrm{Std},s)$ of $\pi$ has no pole at $s=2$.
\end{prop}

\begin{proof}
Let $\GSp_4$ be the symplectic group stabilizing the form on $\set{e_1,e_2,f_2,f_1}$ given by $\gen{e_i,f_j} = \delta_{i,j}$ and all other pairings 0.  Let $P$ and $Q$ respectively denote the Siegel and Klingen parabolic subgroups stabilizing $\gen{f_1,f_2}$ and $\gen{f_2}$, respectively.  Write
\[I^*(\phi,s,w,f_s,\varphi_w) = \int_{\GSp_4(\mathbf{Q})Z(\mathbf{A}) \backslash \GSp_4(\mathbf{A})} \phi(g) E_P^*(g,s,f_s)E_Q^*(g,w,\varphi_w)dg,\]
where $\phi \in \pi$,
\[E_P^*(g,s,f_s) = \sum_{\gamma \in P(\mathbf{Q}) \backslash \GSp_4(\mathbf{Q})} f_s^*(\gamma g), \textrm{ and }E_Q^*(g,w,\varphi_w) = \sum_{\gamma \in Q(\mathbf{Q}) \backslash \GSp_4(\mathbf{Q})} \varphi_w^*(\gamma g).\]
The normalized sections used here are obtained by multiplying $f_s$ and $\varphi_w$ by the zeta factors $\zeta(6s-2)\zeta(3s)$ and $\zeta(4w)$, respectively.  By \cite[Theorem 1.2 and Proposition 1.3]{bfg}, we may choose the data so that we have
\[I^*(\phi,s,w,f_s,\varphi_w) = (*)L^S(\pi,\mathrm{Std},3s-1)L^S(\pi,\mathrm{Spin},2w-\frac{1}{2}),\]
where (*) is nonvanishing at $(1,w_0)$ for a given choice of $w_0 \in \mathbf{C}$.  So it suffices to check that $I^*(\phi,s,w)$ has no pole at $s=1$.  If there is such a pole, then using the argument from the proof of Proposition \ref{prop:poles}, there exists data for which
\[\int_{\Sp_4(\mathbf{Q})Z(\mathbf{A}) \backslash \Sp_4(\mathbf{A})} \phi(g) E_{P'}^*(g,s,f_s)E_{Q'}^*(g,w,\varphi_w)dg\]
has a pole, where $P' = P \cap \Sp_4$ and $Q' = Q \cap \Sp_4$.  The pole of $E_{P'}^*(g,s,f_s)$ at $s=1$ corresponds to the rightmost pole of the Siegel Eisenstein series, which is well-known to have constant residue.  (See \cite[Proposition 5.4.1]{kr}.)  Then the residue is identified with
\[\int_{\Sp_4(\mathbf{Q})Z(\mathbf{A}) \backslash \Sp_4(\mathbf{A})} \phi(g)E_{Q'}^*(g,w,\varphi_w)dg,\]
which obviously vanishes by cuspidality of $\pi$.
\end{proof}

\bibliography{integralRepnBib}
\bibliographystyle{plain}

\end{document}